\documentclass[11pt,letterpaper,notitlepage,twoside]{article}

\usepackage{titling}
\usepackage{enumitem}
\usepackage{titlesec}

\usepackage[top=4cm, bottom=3.5cm, left=3.8cm, right=3.8cm, columnsep=20pt]{geometry}
\usepackage[dvipsnames]{xcolor}
\definecolor{FrenchPink}{RGB}{255,118,164}
\definecolor{VividMulberry}{RGB}{192,17,215}
\definecolor{DenimBlue}{RGB}{47,60,190}
\usepackage{hyperref}
\hypersetup{
	colorlinks=true,     %Colore le texte plutôt que créer une boîte
	linkcolor=DenimBlue,    %Couleur des liens internes normaux
	citecolor=FrenchPink,     %Couleur des citations bibliographiques
	filecolor=VividMulberry,      %Couleur des URL locaux
	urlcolor=VividMulberry,       %Couleur des URL externes
}
\usepackage{fancyhdr}
\fancypagestyle{first}{
	\pagestyle{fancy}
	\fancyhead[C]{}
	\fancyhead[RO,LE]{}
	\fancyhead[LO,RE]{}
	\fancyfoot[C]{}
	\fancyfoot[L]{{\footnotesize Latest update: \today }}
	
	}

\usepackage{amsmath,tabu}
\usepackage{amsthm}
\usepackage{amssymb}
\usepackage{mathtools}
\usepackage{dsfont}
\usepackage{etoolbox}
\newcommand{\rom}[1]{\uppercase\expandafter{\romannumeral #1\relax}}

\usepackage{tabularx}
\usepackage{graphicx}
\usepackage{caption}
\usepackage{subcaption}
\usepackage{float}
\usepackage{booktabs}
\usepackage{wrapfig}
\usepackage{tikz}
\usepackage{tikz-cd}

\usepackage{soul}

\newtheoremstyle{lem}{}{}{\slshape}{}{\bfseries}{}{.5em}{}
\theoremstyle{lem}
\newtheorem{thm}{Theorem}
\newtheorem{lem}[thm]{Lemma}
\newtheorem*{lem*}{Lemma}
\newtheorem{prop}[thm]{Proposition}
\newtheorem*{prop*}{Proposition}
\newtheorem{cor}[thm]{Corollary}
\newtheorem*{thm*}{Theorem}

\newtheorem*{defn*}{Definition}

\newtheorem{conj}{Conjecture}

\newtheorem{hyp}[conj]{Speculation}

\newtheoremstyle{rem}{}{}{\slshape}{}{\itshape}{.}{.5em}{}
\theoremstyle{rem}
\newtheorem{rem}{Remark}
\newtheorem*{ex*}{Example}

\newtheoremstyle{pr}{}{}{}{}{\scshape}{.}{.5em}{}
\theoremstyle{pr}
\newtheorem*{pr*}{Proof}
\AtEndEnvironment{pr*}{\null\hfill\qedsymbol}

\let\emph\relax % there's no \RedeclareTextFontCommand
\DeclareTextFontCommand{\emph}{\bfseries\em}

\usepackage{newpxtext,newpxmath}
\usepackage[utf8]{inputenc}
\usepackage[T1]{fontenc}

\newcommand{\Z}{\mathds{Z}}
\newcommand{\Q}{\mathds{Q}}
\newcommand{\R}{\mathds{R}}
\newcommand{\C}{\mathds{C}}

\newcommand{\D}{\mathds{D}}
\newcommand{\T}{\mathds{T}}
\newcommand{\Id}{\mathds{1}}
\newcommand{\Ker}{\operatorname{Ker}}

\newcommand{\Hom}{\operatorname{Hom}}

\newcommand{\del}{\partial}

\newcommand{\om}{\omega}

\newcommand{\ra}{\rightarrow}

\newcommand{\eps}{\epsilon}
\newcommand{\la}{\lambda}

\newcommand{\brat}[1]{{\left< #1 \right>}}

\def\mrm#1{{\mathrm{#1}}}
\def\cl#1{{\mathcal{#1}}}

\DeclareMathOperator{\Ab}{{\mrm{Ab}}}

\DeclareMathOperator{\Diff}{\mathrm{Diff}}

\DeclareMathOperator{\Ham}{\mathrm{Ham}}

\DeclareMathOperator{\rank}{\mathrm{rank}}
\DeclareMathOperator{\im}{\mathrm{im}}
\DeclareMathOperator{\id}{\mathrm{id}}

\let\phi\varphi
\let\epsilon\varepsilon
\let\emptyset\varnothing

\def\HR{$H$-rational}
\def\HE{$H$-exact}
\def\Ham{\operatorname{Ham}}\def\Symp{\operatorname{Symp}}
\def\LSymp{\m L{\mathrm{Symp}}} \def\LSympZ{\m L{\mathrm{Symp}_0}} \def\LHam{\m L{\mathrm{Ham}}} \def\LLag{\m L\mathrm{Lag}} \def\Man{\mathsf{SMan}} 
\def\bb#1{\mathds{#1}} \def\m#1{\mathcal{#1}}  
\def\wh#1{\widehat{#1}}  

\makeatletter
\let\@fnsymbol\@arabic
\makeatother

\AtEndDocument{\bigskip{\footnotesize%
  \textsc{M. S. Atallah: Department of Pure Mathematics, University of Sheffield, Hicks Building, Hounsfield
Road, Sheffield, S3 7RH, UK} \par
  \textit{E-mail address}, \texttt{m.s.atallah@sheffield.ac.uk} \par
  \addvspace{\medskipamount}
  \textsc{J.-P. Chass\'e: CRM, Universit\'e de
  	Montr\'eal, C.P. 6128 Succ. Centre-Ville Montr\'eal, QC H3C 3J7, Canada} \par
  \textit{E-mail address}, \texttt{jean-philippe.chasse@umontreal.ca} \par
  \addvspace{\medskipamount}
  \textsc{R. Leclercq: Laboratoire de Math\'ematiques Jean Leray, Nantes Universit\'e, 2 Chemin de la Houssini{\`e}re, BP 92208, Nantes, France} \par
  \textit{E-mail address}, \texttt{remi.leclercq@univ-nantes.fr} \par
    \addvspace{\medskipamount}
    \textsc{E. Shelukhin: D\'epartement de Math\'ematiques et de Statistique, Universit\'e de
Montr\'eal, C.P. 6128 Succ. Centre-Ville Montr\'eal, QC H3C 3J7, Canada} \par
  \textit{E-mail address}, \texttt{egor.shelukhin@umontreal.ca} \par
  \addvspace{\medskipamount}

}}

\begin{document}

\title{Weinstein exactness of nearby Lagrangians and related questions}

\author{Marcelo S. Atallah\thanks{Partially supported by the FRQNT and the Fondation Courtois} , Jean-Philippe Chassé\thanks{Partially supported by the Swiss National Science Foundation (grant number  200021\_204107)} , Rémi Leclercq\thanks{Partially supported by the ANR grant 21-CE40-0002 (CoSy)} , Egor Shelukhin\thanks{Partially supported by NSERC, Fondation Courtois, and an Alfred P. Sloan Fellowship}}

\maketitle

\abstract{
 We address the following problem: if a Hamiltonian diffeomorphism maps a Lagrangian submanifold $L$ to a small Weinstein neighborhood of $L$, is the image necessarily Hamiltonian isotopic to $L$ inside that neighborhood? On the one hand, we show that the question can have a negative answer in any symplectic manifold of dimension at least six. On the other hand, we answer an \textit{a priori} weaker form of the question in the positive in various cases when $L$ satisfies a rationality condition: we prove that the image of $L$ is often exact inside the Weinstein neighborhood. We provide applications to the Lagrangian counterpart of the $C^0$ flux conjecture, to $C^0$-rigidity phenomena of Hamiltonian diffeomorphisms, and to topological properties of spaces of Lagrangians with the same rationality constraint. Moreover, we state and prove cases of an analogue of Viterbo's spectral norm conjecture for non-exact Lagrangians{; in the process, we} make progress on an old question of Viterbo regarding integer difference vectors between points of Lagrangians.
}

%%%%%%%%%%%%%%%%%%%%%%%%%%%%%%%%%%%%%%%%%%%%%%%%%%%

\section{Introduction}
\label{sec:introduction}

This paper aims to study the local topological properties of natural sets of Lagrangians, most notably the Hamiltonian and symplectic orbits of a given Lagrangian $L$, respectively
\begin{align*}
	\LHam(L) &:= \Ham(M) \cdot L = \{ \phi(L) \,|\, \phi\in\Ham(M) \}\,,\\
	\LSympZ(L) &:= \Symp_0(M) \cdot L = \{ \psi(L) \,|\, \psi\in\Symp_0(M) \}\,.
\end{align*}
Here, $\Ham(M)$ is the group of diffeomorphisms that are the time-$1$ map of an isotopy generated by a time-dependent vector field $X_t$ such that $\iota_{X_t}\omega$ is exact for all $t\in [0,1]$. The elements of $\Ham(M)$  are called Hamiltonian diffeomorphisms and are important automorphisms of a symplectic manifold. Hamiltonian diffeomorphisms belong to the identity component $\Symp_0(M,\omega)$ of the group of symplectic diffeomorphisms since they are, by definition, isotopic to the identity through a path of symplectic diffeomorphisms.

Note that for the sake of conciseness, we will refer throughout the paper to a {\it closed connected Lagrangian submanifold of a connected symplectic manifold without boundary} as a ``Lagrangian in a symplectic manifold''.

To enable our study, we first fix a metric $g$ on the underlying manifold $L$. Recall that the Weinstein neighbourhood theorem ensures that there exist $r>0$ and a symplectomorphism $\Psi : D^{*}_{r}L \to \m W_{r}(L)$ from the codisk bundle of $L$ of radius $r$ to a neighbourhood $\m W_{r}(L)$ of $L$ in $M$ which maps the 0-section to $L$.

Therefore, understanding $\LHam(L)$ \emph{locally} is intimately related to the nearby Lagrangian conjecture (or NLC for short), which completely characterizes Lagrangians which are in the Hamiltonian orbit of the 0-section in $T^*L$. Indeed, it states that those are precisely the exact Lagrangians. It is known to hold for $S^{1}$, $S^{2}$~\cite{Hind2004}, $\bb R\mathrm P^{2}$~\cite{HindPinsonnaultWu2016,Adaloglou2022}, and $\bb T^{2}$ \cite{DimitroglouGoodmanIvrii2016}. Without restriction on the diffeomorphism type, the most advanced result in the direction of the NLC states that the natural projection $\pi : T^{*}L \to L$ induces a simple homotopy equivalence between any exact closed Lagrangian and the 0-section \cite{AbouzaidKragh2018}. 
This latter result will play a crucial role in our study of the local structure of $\LHam (L)$.

Inspired by this conjecture, we propose that if $L' \in \LHam(L)$ is close to $L$, then there is an accordingly small Hamiltonian isotopy from $L$ to $L'$. More precisely, we consider the following speculation.

\begin{hyp} \label{conj:strong}
	Let $L$ be a Lagrangian in a symplectic manifold $M$. For every small enough neighbourhood $U$ of $L$, we have the following property. If $L'$ is Hamiltonian isotopic to $L$ in $M$ and $L'\subseteq U$, then there exists a Hamiltonian isotopy $\{\phi_t\}_{t\in [0,1]}$ supported in $U$ such that $\phi_1(L)=L'$.
\end{hyp}

If this holds for $L$, then its Hamiltonian orbit $\LHam(L)$ is locally path connected via Hamiltonian isotopies. Although this statement is new in the Lagrangian context, there are some results towards its Hamiltonian counterpart. More precisely, the group $\mathrm{Ham}_c(M)$ of compactly supported Hamiltonian diffeomorphisms of a symplectic manifold $M$ is locally path connected in the $C^0$ topology (via Hamiltonian isotopies) if $M$ is a closed surface or the open ball $B^{2n}$. The 2-dimensional case case follows from Fathi's work on homeomorphisms preserving a volume form~\cite{Fathi1980} and the folkloric fact that a path of such homeomorphisms on a closed surface can be $C^0$ approximated by a path of symplectomorphisms; see~\cite{Seyfaddini2013} for a proof. The case of the open ball was proved by Seyfaddini, also in~\cite{Seyfaddini2013}. Note, however, that the local path connectedness of $\mathrm{Ham}(M)$ does not yield Speculation~\ref{conj:strong}.\par

A confirmation of Speculation~\ref{conj:strong} is within reach only in the few cases where the NLC is known to hold. However, the following weaker form of the speculation should be much more accessible.\par

\begin{hyp} \label{conj:weak}
	Let $L$ be a displaceable Lagrangian in a symplectic manifold $M$. There exists a neighbourhood $U$ of $L$ with the following property. If $L'$ is Hamiltonian isotopic to $L$ in $M$ and $L'\subseteq U$, then $L\cap L'\neq\emptyset$.
\end{hyp}

Note that, if $U$ is a Weinstein neighbourhood, then the restriction of the symplectic form $\omega$ of $M$ to $U$ admits a primitive $\lambda$, i.e.\ $\omega|_U=d\lambda$. Therefore, inside a Weinstein neighbourhood, we can ask if $L'$ is \emph{exact}, i.e.\ if $\lambda|_{L'}=df$ for some $f:L'\to\R$. By a classic result of Gromov~\cite{Gromov1985}, if $L'$ is exact, then it must intersect $L$. In particular, this means that Speculation~\ref{conj:weak} is true whenever $L$ is already exact in $M$ or whenever $H^1(L;\R)=0$. In practice, we will investigate Speculation~\ref{conj:weak} by trying to prove such local exactness phenomenon.

Theorem~\ref{thm:Counter-Example} below shows that even Speculation~\ref{conj:weak} is too much to hope for in full generality; there exist certain {\it irrational} tori for which it fails.

\medskip

{To give an idea of the methods we use to prove our positive results, we introduce a Viterbo-style conjecture. To do so, we need the following rationality notions.}

\begin{defn*}
	Let $L$ be a Lagrangian of a symplectic manifold $(M,\omega)$.  We say that $L$ is \emph{$H$-rational} in $(M,\om)$ if the group of relative periods $\om(H_2(M,L;\Z)) = \tau_L \Z \subset \R$ is a discrete subgroup. We call $\tau_L \geq 0$ the \emph{$H$-rationality constant} of $L.$ We say that $L$ is \emph{$H$-exact} if $\tau_L = 0$. 
\end{defn*}

With these definitions in mind, we make the following conjecture, which we prove in several cases below.

\begin{conj}
	\label{conj: r Vit}
	Let $L$ be a closed connected manifold. Suppose that $K$ is an $H$-rational Lagrangian inside $D^*_r L$ such that the map $(\pi_K)_*: H_1(K;\R) \to H_1(L;\R)$ induced by $\pi_K = \pi|_K$ is not surjective. {Here, $\pi: T^*L \to L$ denotes the canonical projection.} Then, the $H$-rationality constant $\tau_K$ of $K$ satisfies \[\tau_K \leq C r\] for a constant $C$ depending only on $L$ and the choice of an auxiliary metric on it.
\end{conj}

\begin{rem}
	\label{rmk:homotopy_proof}
	We could ask the same question by replacing $\tau_K$ by the rationality constant of $K$, i.e.\ the constant $\tau\geq 0$ such that $\om(\pi_2(M,L)) = \tau \Z$ when it exists. Likewise, it would be interesting to see if there is a version of Conjecture~\ref{conj: r Vit} {that} also applies to irrational Lagrangians, with the role of $\tau_K$ being replaced by $\hbar(K, J) = \min \{ \om(u) > 0\}$, where the minimum runs over all the non-constant $J$-holomorphic disks $u:\mathbb{D} \to M$ with boundary on $K$ for a fixed compatible almost complex structure $J.$ We leave this question for future work.
\end{rem}

\subsection{Main results}

\vspace{10pt}
The following result provides a counterexample to Speculations~\ref{conj:strong} and~\ref{conj:weak} (and to Speculation~\ref{conj:flux} below). It is not, however, a counterexample to the Viterbo-type conjecture. This shows that there is no hope to prove the above speculations in full generality, and that some rationality condition is required.

\begin{thm} 
	\label{thm:Counter-Example}
	In any symplectic manifold of dimension $2n \geq 6$, there exists a Lagrangian torus whose Hamiltonian orbit
	\begin{enumerate}%[label=(\roman*)]
		\item admits arbitrarily Hausdorff-close disjoint elements,
		\item is not closed in the Hausdorff topology inside the set of Lagrangian tori.
	\end{enumerate}
\end{thm}

{Both claims actually hold for any reasonable notion of $C^{1}$ topology~---~see Section~\ref{sec:counter-example-all-conj} below.}

The existence of such tori follows directly from the characterization of product tori in the Hamiltonian orbit of a given product Lagrangian torus in $\bb C^{n}$ by Chekanov \cite{Chekanov1996} and in large enough balls by Chekanov and Schlenk~\cite{ChekanovSchlenk2016}. We present the details in Section~\ref{sec:counter-example-all-conj} below. \par

Again, we make the crucial observation that these Lagrangian tori are not \emph{rational}, and therefore do not contradict our positive results below. \par

On the other hand, we have positive results related to Conjecture~\ref{conj: r Vit}.

\begin{thm}
	\label{thm:conjecture}
	Conjecture \ref{conj: r Vit} holds for $L=L_0 \times L_1\times\dots\times L_k$, where $H_1(L_0;\R)=0$, and $L_i$, $i \geq 1$, satisfies $H_1(L_i;\R)=\R$ and admits a Lagrangian embedding in a Liouville domain $W_i$ with $SH(W_i)=0.$
	Furthermore, if {$L$ is any manifold covered} by a closed connected manifold $L'$ for which Conjecture \ref{conj: r Vit} holds, then it also holds for $L.$
\end{thm}

{Note that, in the above statement, $L_0$ may be a point.}

\begin{rem} \label{rem:conj-Virt_homotopy}
	The prototypical example covered by Theorem \ref{thm:conjecture} is when $L$ has the diffeomorphism type of $L_0\times \T^m$ with $H_1(L_0;\R)=0$. In fact, in this case, an alternative argument can be used to prove a stronger version of the conjecture where $K$ is only assumed to be rational on disks and $(\pi_k)_*$ is not surjective on homology; this is proven in Section \ref{subsec:proof_thm-conj-torus}, as Theorem \ref{thm:conjecture_torus}.
\end{rem}

{We now study Speculations~\ref{conj:strong} and~\ref{conj:weak}.} More precisely, we show that we have \emph{Weinstein neighbourhoods of exactness} for Hamiltonian isotopic Lagrangians. \par

%\begin{thm}
%	\label{thm:speculation-B}
%	Suppose that $L$ is a {\HR} Lagrangian submanifold of $(M,\omega)$ for which Conjecture~\ref{conj: r Vit} holds. Then, there exists a Weinstein neighbourhood $\m W$ of $L$ such that all $L' \in \LHam(L)$ included in $\m W$ is exact in it. \par
%	
%	In particular, Speculation~\ref{conj:weak} has a positive answer for any such manifold, e.g.\ {\HR} $n$-tori, and Speculation~\ref{conj:strong} has a positive answer for {\HR} 2-tori.
%\end{thm}

\begin{thm}
	\label{thm:speculation-B}
	Suppose that $L$ is a {\HR} Lagrangian submanifold of $(M,\omega)$ for which Conjecture~\ref{conj: r Vit} holds. Then, there exists a Weinstein neighbourhood $\m W$ of $L$ such that all $L' \in \LHam(L)$ included in $\m W$ is exact in it. In particular, Speculation~\ref{conj:weak} has a positive answer in this case. \par
\end{thm}

When combined with the NLC, Theorem~\ref{thm:speculation-B} provides a positive answer to Speculation~\ref{conj:strong} for {\HR} 2-tori.

In ongoing work \cite{AtallahShelukhin2024}, a result of local exactness, similar to Theorem~\ref{thm:speculation-B}, is obtained for graphs in $M\times M$ of $C^0$-small Hamiltonian diffeomorphisms for $M$ closed. In that case, there is no requirement that $M$ be rational, contrary to the setup of the present work.

\subsection{Related questions} \label{subsec:related-questions}

\vspace{10pt}

\subsubsection{A question of Viterbo}
An old question of Viterbo is the following. Is there a constant $A\geq 1$ such that every rational Lagrangian $K$ inside $\C^n$ with rationality constant $\tau_K > A$ must have two distinct points whose difference vector has {Gaussian} integer coordinates? {This is equivalent to finding the maximal rationality constant of a nullhomotopic rational Lagrangian in the standard torus $\C^n/(\Z^n+i\Z^n)$.}

By inspecting the proof of the homotopy version of Theorem~\ref{thm:conjecture}, we prove a stronger statement with the difference vector being also real, under the additional assumption that $K$ is contained in $\R^n + i X$ for $X \subset \R^n$ given by $X = \{ (p_1,\ldots, p_n)\;|\; \min_j |p_j| \leq 1\}.$

\begin{prop}\label{cor:Vit numbers}
	Let $K$ be a rational Lagrangian inside $\C^n$ with rationality constant $\tau_K > 2.$ If $K$ is contained in $\R^n + i X$, then there exist two distinct points on $K$ whose difference vector is real with integer coordinates. 
\end{prop}

\subsubsection{Lagrangian flux conjectures}
We now move on to another speculation about Lagrangian submanifolds.

\begin{hyp}[Lagrangian $C^0$ flux conjecture] \label{conj:flux}
	Let $L$ be a Lagrangian in a symplectic manifold $M$. Its Hamiltonian orbit $\LHam(L)$ is Hausdorff-closed in the space $\LLag(L)$ of all Lagrangians that are Lagrangian isotopic to $L$.
\end{hyp}

As far as the authors know, this version of the conjecture has not been studied previously~---~we will talk about its $C^1$ cousin, which has been studied, below. The name that we give it here is in analogy to the famous $C^0$ flux conjecture for Hamiltonian diffeomorphisms, which states that the group $\mathrm{Ham}(M)$ of Hamiltonian diffeomorphisms of a closed symplectic manifold $M$ is $C^0$ closed in the identity component $\mathrm{Symp}_0(M)$ of the group of symplectomorphisms of $M$. This conjecture is only known to hold in some fairly specific cases~\cite{LalondeMcDuffPolterovich1998,Buhovsky2015}, and further results in this direction will appear in \cite{AtallahShelukhin2024}, which still do not resolve this question completely. This is in stark contrast with its $C^1$ cousin, which is known to hold in full generality~\cite{Ono2006}. 
\par

The Lagrangian $C^0$ flux conjecture does not imply the one for Hamiltonian diffeomorphisms. Indeed, it only implies that, whenever $\phi$ is $C^0$-small, $\operatorname{graph}\phi$ is Hamiltonian isotopic to the diagonal in $M\times M$. However, that Hamiltonian isotopy may not be through graphical Lagrangians.
For a similar reason, even if Speculation~\ref{conj:strong} holds for all graphs in $M\times M$, it does not imply the local path connectedness of $\mathrm{Ham}(M)$.
\par

To study this speculation, we can apply the techniques developed to prove Theorems~\ref{thm:nbhd-of-H-exactness} and~\ref{thm:from-H-exact-to-exact}, to provide positive results in a few listed examples.

\begin{thm} \label{thm:C0-flux}
	Suppose that $L$ is either a Lagrangian circle, 2-sphere, or projective plane, or a {\HR} Lagrangian 2-torus. Its Hamiltonian orbit $\LHam(L)$ and symplectic orbit $\LSympZ(L)$ are Hausdorff-closed in the space of Lagrangians diffeomorphic to $L$.
\end{thm}

\begin{rem} \label{rem:variants-flux}
	If we replace the Hausdorff topology by the $C^1$ one, we can fully remove the constraint on the diffeomorphism type of $L$ and still get the analogous result. Furthermore, similar techniques can be utilized to better understand the space of all Lagrangian of a given {\HR ity} constant. We will explore this further in Section~\ref{sec:discussion}.
\end{rem}

\subsubsection{$C^{0}$ rigidity of Hamiltonian diffeomorphisms}
There is a natural variant of Speculation~\ref{conj:weak} where we ask not that $L'$ be close to $L$, but rather that the Hamiltonian diffeomorphism sending $L$ to $L'$ be small. More precisely, we can make the following conjecture. \par

\begin{conj} \label{conj:ham-lag}
	Let $L$ be a displaceable Lagrangian in a symplectic manifold $M$. There exists $\delta>0$ with the following property. If $\phi$ is a Hamiltonian diffeomorphism of $M$ and $d_{C^0}(\Id,\phi)<\delta$, then $L\cap \phi(L)\neq\emptyset$.
\end{conj}

In other words, any Hamiltonian diffeomorphism displacing $L$ is uniformly $C^{0}$-bounded away from 0.

The existence of such a bound is not at all trivial: if $L$ is a displaceable $n$-dimensional submanifold which is not Lagrangian, then it can be displaced by an arbitrarily $C^0$-small Hamiltonian diffeomorphism~\cite{LaudenbachSikorav1994}. Moreover, this does not follow from the fact that Lagrangians have positive displacement energy, since there are Hamiltonian diffeomorphisms which are arbitrarily $C^0$-small, but arbitrarily Hofer-large.

However, this is not expected to be the case for the spectral metric, that is, $C^0$-small Hamiltonian diffeomorphisms should also have small spectral norm. More precisely, Conjecture~\ref{conj:ham-lag} follows from the fact that Lagrangians have positive spectral displacement energy~\cite{KislevShelukhin2022} in the cases where it is known that the spectral metric is $C^0$-continuous, i.e.\ when $M$ is $\C^n$~\cite{Viterbo1992}, a closed surface~\cite{Seyfaddini2013}, closed and symplectically aspherical~\cite{BuhovskyHumiliereSeyfaddini2021}, $\C P^n$~\cite{Shelukhin2022}, or closed and negative monotone~\cite{Kawamoto2022} or symplectically Calabi-Yau~\cite{SWilkins}.\par

In the context of this paper, Conjecture~\ref{conj:ham-lag} is implied by Speculation~\ref{conj:strong} or by the Hamiltonian version of Speculation~\ref{conj:weak} above when it holds. However, it turns out we can directly prove a much stronger statement.
\begin{prop} \label{prop:ham-lag}
	Let $L$ be a displaceable Lagrangian of $M$, and take a compact $K$ containing $L$ in its interior. Suppose furthermore that one of the following holds:
	\begin{enumerate}%[label=(\alph*)]
		\item $M$ is rational, i.e.\ $\omega(\pi_2(M))$ is discrete, or
		\item the image of $\pi_1(L)\to\pi_1(M)$ is torsion.
	\end{enumerate}
	Then, there exists $\delta>0$ with the following property. If $\phi$ is a Hamiltonian diffeomorphism supported in $K$ and $d_{C^0}(\Id,\phi)<\delta$, then $L\cap \phi(L)\neq\emptyset$.
\end{prop}

\begin{rem} \label{rem:app_ham-lag}
	As will be explored in Section \ref{subsec:proof_prop-ham-lag}, this has further applications for understanding the image of a {\HR} Lagrangian under $C^0$-limits of symplectomorphisms. This shows yet more $C^0$-rigidity phenomena for these Lagrangians.
\end{rem}

\subsubsection{A conjecture of Buhovsky and Opshtein}

Using our Proposition~\ref{prop:estimate_taut-form}, which is essentially implicit in \cite{MembrezOpshtein2021}, we settle a conjecture of Buhovsky and Opshtein \cite[Conjecture 3]{BO2016} previously known only for $L' = T^n$ by \cite[Proposition 5.2]{BO2016}. Namely, in Section \ref{sec:proof-theorems-From-H-exact-2-exact} we prove the following.

\begin{thm}\label{thm: BO}
	Let $L, L'$ be closed manifolds and $i_k: L \to T^*L',$ $k \geq 1,$ be a sequence of smooth Lagrangian embeddings which $C^0$-converges, as mappings, to a topological embedding $i:L \to T^*L'$ with $i(L) = 0_{L'}.$ Then $[i_k^*\lambda] \to 0$ in $H^1(L;\R),$ where $\lambda$ is the tautological one-form on $T^*L'.$ 
\end{thm}

In fact our proof applies in the slightly more general case where $i:L \to T^*L'$ is a continuous map with $i(L) = 0_{L'}$ and the resulting map $i: L \to 0_{L'}$ is a homotopy equivalence.

\subsubsection{Topological properties of $\LHam(L)$ in the Hausdorff topology}
{While we need the NLC in all cases where we can prove Speculation~\ref{conj:strong}, the fact that $\LHam(L)$ is locally path connected in the Hausdorff topology can be proved whenever Conjecture~\ref{conj: r Vit} holds. More precisely, we get the following.}

\begin{prop} \label{prop:LHam-loc-contractible}
	Suppose that $L$ is {\HR} and that Conjecture~\ref{conj: r Vit} holds for it. Then $\LHam(L)$ is locally contractible in the Hausdorff topology.
\end{prop}

\section{Irrational counterexamples}
\label{sec:counter-example-all-conj}

We now explain the construction of the Lagrangian tori from Theorem~\ref{thm:Counter-Example}. These tori support the fact that we need in general to require some type of \emph{rationality} condition on our Lagrangians for Speculations~\ref{conj:strong}, \ref{conj:weak} or \ref{conj:flux} to hold.

\medskip
We start with the case when $M=\C^3$. Consider the product torus $L=T(1,2,1+\alpha):=T(1)\times T(2)\times T(1+\alpha)$, where $\alpha>0$ is an irrational number and $T(A)\subseteq \C$ denotes the round circle enclosing area $A>0$. By work of Chekanov~\cite{Chekanov1996}, another product torus $T(a,a+b,a+c)$ with $a,b,c>0$ is Hamiltonian isotopic to $L$ in $\C^3$ if and only if $a=1$ and $\mathrm{span}_\Z\{b,c\}=\mathrm{span}_\Z\{1,\alpha\}=:G$. \par

Let $p_i/q_i$ be the $i$th convergent to $\alpha$ obtained from its infinite continued fraction. In particular $p_i, q_i$ are coprime and $|p_i-q_i \alpha| < \frac{1}{q_{i+1}}.$ Fix $\epsilon>0.$ Then for all $i \geq i_0$ sufficiently large $|p_i-q_i \alpha| < \epsilon.$ Moreover, the matrix
\begin{align*}
	\begin{pmatrix}
		p_i & q_i \\ p_{i+1} & q_{i+1}
	\end{pmatrix}
\end{align*}
has determinant $\pm 1$ and hence for $i \geq i_0,$ $b =p_i-q_i \alpha,$ $c = p_{i+1} - q_{i+1}\alpha$ satisfy $\mathrm{span}_\Z\{b,c\}=G$ while also $|b|<\epsilon, |c|<\epsilon.$ By changing the signs of $b,c$ if necessary, which preserves both conditions, we can ensure that $b>0, c>0.$ 

This means that we can take $T(1,1+b,1+c)$ which are all in the Hamiltonian orbit of $L$ but are arbitrarily $C^1$-close to the monotone torus $T(1,1,1)$. This proves \textit{(ii)} of Theorem~\ref{thm:Counter-Example} in $\bb C^{3}$ and shows that, without the {\HR} hypothesis on $L$, not even the Lagrangian $C^1$ flux conjecture is true in $\C^3$. \par

Note that a similar argument as above actually implies that the set of $b,c>0$ such that $T(1,1+b,1+c)$ is Hamiltonian isotopic to $L$ is dense in $\R^2_{>0}$. This means that any neighbourhood $U$ of such a torus $T(1,1+b,1+c)$ contains infinitely many $T(1,1+b',1+c')$ in the same Hamiltonian orbit. But $T(1,1+b,1+c)\cap T(1,1+b',1+c')=\emptyset$ if $b\neq b'$ or $c\neq c'$. This proves \textit{(i)} of Theorem~\ref{thm:Counter-Example} in $\bb C^{3}$  and shows that, without the {\HR} hypothesis on $L$, not even Speculation~\ref{conj:weak} is true in $\C^3$. {Note that this also shows that the Lagrangian flux group of Lagrangian isotopies need not be discrete since that of $T(1,2,1+\alpha)$ is not.}  \par

We now explain how to generalize the result to any manifold of dimension $2n\geq 6$. First note that by taking a product with $T(1)^{n-3}$, we get a counterexample to our speculations in $\C^n$ whenever $n\geq 3$. Furthermore, by Theorem~1.1(ii) of~\cite{ChekanovSchlenk2016}, the Hamiltonian isotopy from $T(1,\dots,1,1+b,1+c)$ to $T(1,\dots,1,1+b',1+c')$ can be taken to be fully supported in the ball $B^{2n}(A)$ of capacity $A=n+1+\max\{b+c,b'+c'\}$, i.e.\ of radius $\sqrt{\frac{A}{\pi}}$. In particular, for $b$, $b'$, $c$, and $c'$ small enough, it can be supported in the ball of capacity $n+2$. Therefore, we get a counterexample in $M=B^{2n}(n+2)$ But then, by simply rescaling the ball, we get a counterexample in the ball $B^{2n}(A)$ for any $A>0$. By the Darboux theorem, any symplectic manifold $M^{2n}$ admits a symplectic embedding of the ball $B^{2n}(A)$ for $A$ small enough, which gives the counterexample for every $M$ with $\dim M\geq 6$. \par

\begin{rem} \label{rem:counterexample_4dim}
	Interestingly enough, the above counterexample does not work in dimension 4. Indeed, Chekanov's classification of product tori implies that every product torus $L$ in $\C^2$ has a $C^1$ neighbourhood $U$ such that $\LHam (L)\cap U=\{L\}$. In particular, the $C^1$ version of Speculation~\ref{conj:strong} holds for $L$, and if its Hamiltonian orbit is not closed, then the limit cannot be a product or a Chekanov torus. By Theorem~1.3 of~\cite{ChekanovSchlenk2016}, the same holds for product tori in small enough Darboux balls in subtame symplectically aspherical symplectic 4-manifolds. \par
	
	However, we expect that one can use Theorem~1.5 of~\cite{ChekanovSchlenk2016} to construct~---~in a similar fashion as above~---~a counterexample to Speculation~\ref{conj:weak} in any (spherically) irrational symplectic 4-fold. 
\end{rem}

\section{Relations between homological rationality and exactness and their standard counterparts}
\label{sec:central-lemma-and-H2pi}

In this section, we discuss relations between standard rationality/exactness and {\HR}ity/{\HE}ness. In Section~\ref{sec:central-lemma}, we prove the general following fact: for a closed Lagrangian of the cotangent bundle, being {\HE} is equivalent to being isotopic to an exact Lagrangian. In Section~\ref{sec:from-HE-exact}, we explain some specific situations in which {\HR}ity reduces to rationality. Finaly in Section~\ref{sec:why-h-rationality}, we give an example which illustrates why we generally work with the {\HR}ity condition rather than the standard rationality one. \par

\subsection{The characterization lemma}\label{sec:central-lemma}
\vspace{10pt}

The following lemma will prove to be quite useful in order to prove the main results of this work.
\begin{lem}
	\label{lemma:central}
	Let $L$ be a closed Lagrangian in $T^{*}Q$, the following are equivalent:
	\begin{enumerate}%[label=(\roman*)]
		\item $L$ is {\HE};
		\item the Liouville form $[\lambda_0|_L]$ is in the image of the map in $H^1(\cdot;\R)$ induced by the composition  $L\to T^*Q\to Q$;
		\item the composition $L\to T^*Q\to Q$ is a homotopy equivalence;
		\item $L$ is isotopic to an exact Lagrangian through Lagrangian submanifolds;
		\item $L$ is symplectically isotopic to an exact Lagrangian;
		\item $L$ is a shift of an exact Lagrangian by a closed one-form on $Q$.	
	\end{enumerate}
\end{lem}
\begin{proof}
	{We first show that \textit{(i)} and \textit{(ii)} are equivalent. To see this, note that {\HE ness} is equivalent to $\lambda_0|_L$ vanishes on the image of $H_2(T^*Q,L;\R)\to H_1(L;\R)$. But this is equivalent to $\delta^*[\lambda_0|_L]=0$, where $\delta^*:H^1(L;\R)\to H^2(T^*Q,L;\R)$ is the coboundary operator. By the long exact sequence in cohomology, this is equivalent to $[\lambda_0|_L]$ being in the image of the map induced by the inclusion $i:L\hookrightarrow T^*Q$. We then conclude using the fact that the projection $\pi: T^*Q\to Q$ is a homotopy equivalence.}
	
	{Suppose we have \textit{(ii)}.} This means that there exists a closed 1-form $\sigma \in \Omega^{1}(Q)$ such that $[i^{*}\lambda]=[i^{*}(\pi^{*}\sigma)]$. Now, $\sigma$ induces a fibrewise symplectomorphism $\psi_{\sigma}$ of $T^{*}Q$ which satisfies $[\psi_{\sigma}^{*}(i^{*}\lambda)]=0$ so that $\psi_{\sigma}$ maps $L$ to an exact Lagrangian. {Furthermore, $\psi_\sigma$ generates a shift by $\sigma$.} This shows that \textit{(ii)} yields \textit{(vi)}, which obviously yields \textit{(v)} and thus \textit{(iv)}.

	Note also that \textit{(iv)} implies by definition that the inclusion $i: L\hookrightarrow T^*Q$ is homotopic to the inclusion of an exact Lagrangian. Indeed, $L$ being Lagrangian isotopic to $L'$ means precisely that there is a smooth map $F:[0,1] \times L \to T^*Q$ such that $F_0 = i,$  $F_t$ for all $t$ is a Lagrangian embedding, and $F_1(L) = L',$ where $F_t: L \to T^*Q$ is given by $F_t(x) = F(t,x).$ But, when $L$ is exact, the composition $L\to T^*Q\to Q$ is a (simple) homotopy equivalence~\cite{AbouzaidKragh2018}, i.e. \textit{(iii)} holds. {Finally, note that \textit{(iii)} directly implies \textit{(ii)}.}
\end{proof}

\begin{rem}
	The notions \textit{(vi)} and \textit{(ii)} appear in \cite{AtallahShelukhin2024} under the names of {\em almost exact} and {\em $H_1$-standard} Lagrangians in $T^*Q$, respectively.
\end{rem}

\subsection{When does {\HR}ity reduce to rationality?}
\label{sec:from-HE-exact}
Obviously, {\HR ity} implies usual rationality, i.e.\ $\omega(H_2(M,L))$ being discrete implies that $\omega(\pi_2(M,L))$ also is. Furthermore, in many cases, these conditions are equivalent. This is the case, for example, when $\pi_1(M)=0$. Indeed, in this case, the relative Hurewicz morphism $\pi_2(M,L)\to H_2(M,L;\Z)$ can be shown to be surjective. Expanding on this idea, we get the following.

\begin{lem} \label{lem:omega_pi-vs-H}
	Suppose that $[\pi_1(M),\pi_1(M)]$ is finite. Then, we have that
	\begin{align*}
		N\omega(H_2(M,L;\Z))\subseteq\omega(\pi_2(M,L))+\omega(H_2(M;\Z)),
	\end{align*}
	where $N$ is the order of $[\pi_1(M),\pi_1(M)]$. In particular, if $\pi_1(M)$ is abelian, then we have equality.
\end{lem}
\begin{pr*}
	We consider the following commutative diagram.
	\begin{center}
		\begin{tikzcd}
			\pi_2(M) \arrow[r, "j"] \arrow[d, "h_2"] & \pi_2(M,L) \arrow[r, "\del"] \arrow[d, "h''_2"] & \pi_1(L) \arrow[r, "i"] \arrow[d, "h'_1"] & \pi_1(M) \arrow[d, "h_1"] \\
			H_2(M) \arrow[r, "j"] & H_2(M,L) \arrow[r, "\del"]  & H_1(L) \arrow[r,"i"]  & H_1(M)
		\end{tikzcd}
	\end{center}
	Here, the rows are the long exact sequences of the pair $(M,L)$ in homotopy and homology with integer coefficients, respectively, and the columns are the various Hurewicz morphisms; it commutes by naturality of the Hurewicz map. \par
	
	The proof follows from a straightforward diagram chasing argument, but we still give the details. Let $A\in H_2(M,L)$. Since $h'_1$ is surjective~---~the Hurewicz morphism in first degree is simply the abelianization morphism~---~there is some $a\in\pi_1(L)$ such that $\del (A)=h'_1(a)$. But note that 
	\begin{align*}
		h_1i(a)=ih'_1(a)=i\del(A)=0.
	\end{align*}
	Therefore, $i(a)\in\Ker h_1=[\pi_1(M),\pi_1(M)]$. By hypothesis, $i(Na)=0$. Therefore, there is some $u\in \pi_2(M,L)$ such that $\del(u)=Na$. But note that
	\begin{align*}
		\del(NA-h_2''(u))=h'_1(Na)-h'_1(\del u)=0.
	\end{align*}
	There is thus some $B\in H_2(M)$ such that $NA=j(B)+h_2''(u)$. To conclude, we only note that $\omega(j(B))=\omega(B)$ and $\omega(h''_2(u))=\omega(u)$.
\end{pr*}

Note that this corollary recovers the statement at the start of the subsection that {\HR ity} and rationality are the same when $\pi_1(M)=0$. In the case $M=D^*_rL$, the condition that $\omega(H_2(M;\Z))=0$ is automatically satisfied since $\omega$ is exact, and the condition on the commutator subgroup of $\pi_1(M)$ becomes that $[\pi_1(L),\pi_1(L)]$ be finite. {Therefore, we get directly the following result.}

\begin{cor} \label{cor:r-Vit_homotopy-vs-homology}
	{Suppose that $[\pi_1(L),\pi_1(L)]$ is finite. Then, Conjecture~\ref{conj: r Vit} is equivalent to its homotopical version, i.e.\ the one where $K$ is only rational, and we get a bound on its rationality constant.} 
\end{cor}

\begin{rem} \label{rem:here_vs_thm-conj-torus}
	{The homotopical version of Theorem~\ref{thm:conjecture}, that is, Theorem~\ref{thm:conjecture_torus} below, proves the homotopical version of Conjecture~\ref{conj: r Vit} when $L=L_0\times\T^m$ with $H_1(L_0;\R)=0$. This intersects with the above corollary precisely in the case when $\pi_1(L_0)$ is finite.} 
\end{rem}

Using Corollary~\ref{cor:r-Vit_homotopy-vs-homology}, we can obtain a homotopical version of Theorem~\ref{thm:nbhd-of-H-exactness}. However, the proof of Theorem~\ref{thm:from-H-exact-to-exact} truly requires {\HE ness}. One exception to this is when $\pi_2(M,L)\to\pi_1(L)$ is surjective, but an argument similar to the above shows that rationality is then equivalent to {\HR ity} as long as $q\omega(H_2(M;\Z))= \omega(\pi_2(M,L))$ for some $q\in\Q$, which is the case generically (and is always the case if $[\pi_1(L),\pi_1(L)]$ is finite).

\subsection{Why {\HR}ity in general?}
\label{sec:why-h-rationality}
We end this section with an example which showcases the need to work with {\HR} Lagrangians and not just rational Lagrangians for many applications. Note that this example is such that $[\pi_1(L),\pi_1(L)]$ is not finite.

In~\cite{Polterovich1993}, Polterovich constructs for any vector $v\in \R^n$ and any flat manifold $Q$ a Lagrangian torus $L_v$ in $T^*Q$. This torus has the property that
\begin{enumerate}%[label=(\roman*)]
	\item for a contractible open $U\subseteq Q$, $L_v\cap T^*Q|_U=U\times\{v\}\subseteq U\times\R^n$;
	\item the map $L_v\to Q$ given by restriction of $\pi:T^*Q\to Q$ is a covering.
\end{enumerate}
We concentrate our efforts on the simplest case: $n=2$ and $Q=K$ is the Klein bottle. In that case, $L_v\to K$ is the 2:1 cover. \par

First note that $L_v$ is weakly exact in $T^*K$. To see this, denote by $p:\T^2\to K$ the 2:1 cover and take $\widetilde{p}:T^*\T^2\to T^*K$ to be its lift using the flat metrics on $\T^2$ and $K$. Point (i) gives that $\widetilde{p}^{-1}(L_v)=\T^2\times\{v\}\subseteq T^*\T^2=\T^2\times\R^2$. But any disk $u$ with boundary along $L_v$ admits a lift $\widetilde{u}$ in $T^*\T^2$ with boundary along $\T^2\times\{v\}$. Since $\T^2\times\{v\}\hookrightarrow T^*\T^2$ is a homotopy equivalence, $\pi_2(T^*\T^2,\T^2\times\{v\})=0$, and the lift $\widetilde{u}$ is contractible. But then, so must be $u$, and we have that $\pi_2(T^*K,L_v)=0$. \par

On the other hand, $L_v$ is \emph{not} {\HE}. Indeed, let $\gamma:S^1\to K$ be a loop admitting a lift to $L_v$, that is, $[\gamma]\in p_*(\pi_1(\T^2))$. Since $L_v\to K$ is a 2:1 cover, there are two lifts $\widetilde{\gamma}_1$ and $\widetilde{\gamma}_2$ of $\gamma$. Furthermore, each lift $\widetilde{\gamma}_i$ defines a cylinder $C_i$ in $T^*K$ by taking $C_i(s,t)=t\widetilde{\gamma}_i(s)$, $(s,t)\in S^1\times [0,1]$. Note that $\del C_i=\widetilde{\gamma}_i\sqcup -\gamma$, where the minus sign denotes the reversal of orientation. Therefore, $C:=C_1\cup_\gamma -C_2$ is a cylinder in $T^*K$ with boundary along $L_v$. Furthermore, it has area
\begin{align*}
	\omega_0(C)=\lambda_0(\widetilde{\gamma}_1)-\lambda_0(-\widetilde{\gamma}_2)=2\int_{S^1}\langle v,\dot{\gamma}(s)\rangle ds,
\end{align*}
where $\langle\cdot,\cdot\rangle$ denotes the Euclidean scalar product. In particular, if we take $\gamma$ to be a simple loop corresponding to a straight line in the fundamental domain of $\R^2$ defining $K=\R^2/\pi_1(K)$ and $v$ to be positively proportional to $\dot{\gamma}$, then $\omega_0(C)=2|v|>0$. Therefore, such an $L_v$ is indeed not {\HE}. \par

Finally, note that, as $v\to 0$, $L_v\to K$ in the Hausdorff topology. Therefore, however small we take a neighbourhood of the zero-section of $T^*K$, there is a weakly exact Lagrangian in that neighbourhood which is not exact. {Therefore, there is no homotopical equivalent of Theorem~\ref{thm:from-H-exact-to-exact} with $\LHam(L)$ replaced by the space of all $\tau$-rational Lagrangians if $[\pi_1(L),\pi_1(L)]$ is not finite.} In particular, many applications in the introduction do not have equivalents in spaces of $\tau$-rational Lagrangians.

\section{Proof of Theorem~\ref{thm:conjecture}} \label{sec:nbhd-of-H-exactness}

We now turn our attention to Theorems~\ref{thm:conjecture} and~\ref{thm:speculation-B}. More precisely, we will prove the following weaker form of Theorem~\ref{thm:speculation-B}.
\begin{thm}
	\label{thm:nbhd-of-H-exactness}
	Suppose that $L$ is a Lagrangian submanifold of $(M,\omega)$ for which Conjecture~\ref{conj: r Vit} holds. Then, for each $\tau\geq 0$, there exists a Weinstein neighbourhood $\m W(L)$ of $L$ such that all $\tau$-{\HR} $L'$ included in $\m W(L)$ is {\HE} in $\m W(L)$.
\end{thm}

To obtain a proof, we first introduce some capacities inspired by work of Cieliebak and Mohnke~\cite{CieliebakMohnke2018} (Section~\ref{subsec:capacities}), and we explain how their finiteness implies Theorem~\ref{thm:conjecture} (Section~\ref{subsec:finiteness-capacities}). We show that this straightforwardly yields Theorem~\ref{thm:nbhd-of-H-exactness} (Section~\ref{subsec:proof-nbhd-H-exactness}). We then prove that Conjecture~\ref{conj: r Vit} is closed under covering, which concludes the proof of Theorem~\ref{thm:conjecture} (Section~\ref{subsec:conj-r-Vit_corverings}). Finally, we state and prove its homotopical variant (Section~\ref{subsec:proof_thm-conj-torus}). Note that the methods developed here will also be central to the proof of the results announced in Section~\ref{subsec:related-questions}.

\subsection{Some capacities \textit{\`a la} Cieliebak--Mohnke} \label{subsec:capacities}
In~\cite{CieliebakMohnke2018}, Cieliebak and Mohnke introduce~---~and compute in some cases~---~a capacity which measures, in a given domain, the largest possible area of a minimal disk with boundary along a Lagrangian torus. We start by introducing a small tweak in their definition, which will turn out to be quite useful in our setting. \par

Let $Q$ be a closed connected $n$-manifold. For any $2n$-dimensional symplectic manifold $(X,\omega)$, we define {three} classes of Lagrangians:
\begin{align*}
	\mathscr{L}_Q(X)
	&:=\{L=\mathrm{Im}(f:Q\hookrightarrow X)\ |\ f^*\omega=0,\ \omega(H_2(X,L;\Z))\neq 0\} \\
	{\mathscr{L}_Q^{'}(X)} 
	&:=\{L=f(Q)\in\mathscr{L}_Q(X) \ |\  \im(H_1(f)\otimes\R)\neq H_1(X;\R) \mbox{ or } H_1(X;\R)=0 \}\\
	\mathscr{L}_Q^0(X)
	&:=\{L=f(Q)\in\mathscr{L}_Q(X) \ |\  H_1(f)\otimes\R= 0\}
\end{align*}
where $H_1(f)\otimes\R$ is the map induced by $f$ on first homology with real coefficients.

\begin{lem}\label{lem:relation-between-curlyLs}
	We always have the inclusions $\mathscr{L}^{0}_Q(X)\subset\mathscr{L}^{'}_Q(X)\subset\mathscr{L}_Q(X)$.
	
	Moreover, we have the following particular cases.
	\begin{itemize}
		\item If $H_1(X;\R)=0$, we have $\mathscr{L}^{0}_Q(X) = \mathscr{L}^{'}_Q(X) = \mathscr{L}_Q(X)$.
		\item 	If $\dim H_1(Q;\R)=1$ and $X$ is exact, we have $\mathscr{L}^{0}_Q(X) = \mathscr{L}^{'}_Q(X) = \mathscr{L}_Q(X)$.
		\item  If $\dim H_1(Q;\R)\leq \dim H_1(X;\R)$ and $X$ is exact, we have $\mathscr{L}^{'}_Q(X)=\mathscr{L}_Q(X)$.
	\end{itemize}
\end{lem}
\begin{pr*}
	The general statement and the first point are obvious.

	Now, we fix $Q$ and $X$ as in the second point. We can assume that there is a Lagrangian embedding $f:Q\hookrightarrow X$; otherwise all sets are empty. Since $\dim H_1(Q;\R)=1$, $H_1(f)\otimes\R$ is either 0 or injective. Suppose that it is injective. By the long exact sequence in homology, we then get that the boundary map $\del:H_2(X,L;\R)\to H_1(L;\R)$ is zero, where $L=f(Q)$. Since $\omega(H_2(X,L;\R))=\lambda(\del(H_2(M,L;\R)))$ whenever $\omega=d\lambda$, we then conclude that $L$ is {\HE}. In particular, $L\notin\mathscr{L}_Q(X)$. Therefore, we have that $\mathscr{L}_Q(X)=\mathscr{L}^0_Q(X)$.

	Finally, for the third point, $\omega(H_2(X,L;\Z))\neq 0$ implies by the long exact sequence of a pair that $H_1(f)\otimes\R$ has non-trivial kernel and therefore $\dim\im(H_1(f)\otimes\R)\leq \dim H_1(X;\R)-1$.
\end{pr*}
In turn, this defines {three} capacities:
\begin{align*}
	c_Q(X) &:=\sup\{A^H_\mathrm{min}(L,X)\ |\ L\in\mathscr{L}_Q(X)\}\in [0,+\infty]; \\
	c_Q^{'}(X) &:= \sup\{A^H_\mathrm{min}(L,X)\ |\ L\in\mathscr{L}_Q^{'}(X)\}\in [0,+\infty];\\
	c_Q^0(X) &:=\sup\{A^H_\mathrm{min}(L,X)\ |\ L\in\mathscr{L}_Q^0(X)\}\in [0,+\infty],
\end{align*} where
\begin{align*}
	A^H_\mathrm{min}(L,X):=\inf\{\omega(u)\ |\ u\in H_2(X,L;\Z), \omega(u)>0\}.
\end{align*}
We take the convention that $c_Q(X)=0$ (respectively $c_Q^0(X)=0$, {$c_Q^{'}(X)=0$}) if $\mathscr{L}_Q(X)=\emptyset$ (respectively $\mathscr{L}_Q^0(X)=\emptyset$, {$\mathscr{L}_Q^{'}(X)=\emptyset$}). Obviously, we have that $c_Q^0\leq c'_Q\leq c_Q$. Finally, we set
\begin{align*}
	c_\mathrm{all}(X) &:=\sup c_Q(X), \\
	{c_\mathrm{all}^{'}(X)} &:= \sup c_Q^{'}(X), \quad \mbox{and}\\
	c_\mathrm{all}^0(X) &:=\sup c_Q^0(X),
\end{align*}
where the supremum runs over all closed connected $n$-dimensional manifolds.

\begin{rem} \label{rem:Cieliebak-Mohnke}
	The main differences between our definition and Cieliebak--Mohnke's are that we work with homology instead of homotopy, we allow any $Q$ and not only tori, and we only look at Lagrangians which do bound some homology class with nonvanishing area. The latter is central to our argument, as we will mainly be interested in the case $X=D^*Q$, but such a manifold obviously admits an exact Lagrangian $Q$. Therefore, without this restriction, $c_Q(D^*Q)$ would be infinite for trivial reasons, which runs counter to our purpose here. \par
	
	{However, we could develop an entirely analogous theory using homotopy and get some homotopical version of Theorem~\ref{thm:nbhd-of-H-exactness}~---~see Section~\ref{sec:from-HE-exact} for a discussion.}
\end{rem}

The following properties follow directly from the definition of the capacities.
\begin{lem} \label{lem:c_basic-prop}
	Let $c$ denote $c_Q$, {$c'_Q$}, $c^0_Q$, $c_\mathrm{all}$, {$c'_\mathrm{all}$} or $c_\mathrm{all}^0$. We have the two following properties.
	\begin{enumerate}%[label=(\roman*)]
		\item For all $\alpha\neq 0$, we have that $c(X,\alpha\omega)=|\alpha|c(X,\omega)$.
		\item If there is a 0-codimensional symplectic embedding $\iota:X\hookrightarrow X'$ such that $H_2(X',\iota(X);\R)=0$, then $c(X)\leq c(X')$.
	\end{enumerate}
\end{lem}

The problem with the monotonicity property (ii) when $H_2(X',\iota(X);\R)\neq 0$ is that there could then be homology classes in $X'$ with smaller area than those in $X$~---~thus inverting the expected direction of the inequality. However, the capacity $c^0_Q$ partially goes around that issue. \par

\begin{lem} \label{lem:c0_monotonicity}
	If there exists a 0-codimensional symplectic embedding $\iota:X\hookrightarrow X'$ and $X'$ is exact, then $c_Q^0(X)\leq B c_Q^0(X')$, where $B\geq 1$ only depends on the torsion part of $H_1(X;\Z)$. 
\end{lem}

\begin{pr*}
	Let $\lambda'$ be a primitive of the symplectic form of $\omega'$ on $X'$. Then, $\lambda=\iota^*\lambda'$ is a primitive of $\omega$ on $X$. Fix $L=f(Q)\in\mathscr{L}_Q^0(X)$. Since $H_1(f)\otimes\R=0$, we must have that $f_*(H_1(Q;\Z))$ is a torsion subgroup of $H_1(X;\Z)$. Take $B$ to be the order of the torsion of $H_1(X;\Z)$ if it is nonzero, i.e.\ if
	\begin{align*}
		H_1(X;\Z)=\Z^b\oplus\Z_{p_1^{k_1}}\oplus\dots\oplus\Z_{p_\ell^{k_\ell}},
	\end{align*}
	then $B=p_1^{k_1}\dots p_\ell^{k_\ell}$. If $H_1(X;\Z)$ has no torsion, then we simply set $B=1$. We thus have $B\cdot f_*(H_1(Q;\Z))=0$. By the homology long exact sequence of the pair $(X,L)$, this is equivalent to saying that $\del H_2(X,L;\Z)\supseteq B\cdot H_1(L;\Z)$. Therefore, we have that 
	\begin{align*}
		A^H_\mathrm{min}(L,X)
		&=\inf_{\substack{a\in \del H_2(X,L;\Z) \\ \lambda(a)>0}}\lambda(a) \\
		&\leq\inf_{\substack{a\in B\cdot H_1(L;\Z) \\ \lambda(a)>0}}\lambda(a) \\
		&= B\cdot\inf_{\substack{a'\in H_1(\iota(L);\Z) \\ \lambda'(a')>0}} \lambda'(a') \\
		&\leq B\cdot A^H_\mathrm{min}(\iota(L),X').
	\end{align*} 
	Since $\iota(\mathscr{L}_Q^0(X))\subseteq \mathscr{L}_Q^0(X')$, this gives the desired inequality.
\end{pr*}
From Lemma \ref{lem:relation-between-curlyLs} above, we directly get.
\begin{lem} \label{lem:c-c0_equal}
	\begin{itemize}
		\item If $H_1(X;\R)=0$, we have $c_Q(X) =c'_Q(X)= c^0_Q(X)$.
		\item If $\dim H_1(Q;\R)=1$ and $X$ is exact, we have $c_Q(X)=c'_Q(X)=c^0_Q(X)$.
		\item  If $\dim H_1(Q;\R)\leq \dim H_1(X;\R)$ and $X$ is exact, we have $c'_Q(X)=c_Q(X)$.
	\end{itemize}
\end{lem}

We end this short list of properties of our capacities by proving that they behave relatively well under products. \par

\begin{lem} \label{lem:c_products}
	Suppose that $Q'$ admits a {\HE} Lagrangian embedding in $X'$. Then, $c_Q(X)\leq c_{Q\times Q'}(X\times X')$. In particular, $c_\mathrm{all}(X)\leq c_\mathrm{all}(X\times X')$ as soon as $X'$ admits a {\HE} Lagrangian. {The same holds for the corresponding $c'$ capacities} 
	
	If $Q'$ admits any Lagrangian embedding in an exact $X'$ and $H_1(Q';\R)=0$, then we have that $c^0_Q(X)\leq c^0_{Q\times Q'}(X\times X')$. In particular, $c_\mathrm{all}^0(X)\leq c_\mathrm{all}^0(X\times X')$ as soon as $X'$ admits a Lagrangian with vanishing first Betti number.
\end{lem}

\begin{pr*}
	Let $L$ be the image of a Lagrangian embedding of $Q$ in $X$, and let $L'$ be the image of a {\HE} Lagrangian embedding in $X'$. Note that we can suppose that $L$ bounds some homology class, otherwise the inequality is trivial. Let thus $v:(\Sigma,\del\Sigma)\to (X\times X',L\times L')$ for some compact surface $\Sigma$ with boundary. Projecting on each component gives maps $u:(\Sigma,\del\Sigma)\to (X,L)$ and $u':(\Sigma,\del\Sigma)\to (X',L')$. Furthermore, if $\omega$ is the symplectic form of $X$ and $\omega'$ of $X'$, we then have that
	\begin{align*}
		(\omega\oplus\omega')(v)= \omega(u)+\omega(u')=\omega(u),
	\end{align*}
	since $L'$ is $H$-exact. Taking infima over all $v$, we thus get
	\begin{align*}
		c_{Q\times Q'}(X\times X')\geq A^H_\mathrm{min}(L\times L',X\times X')=\inf_{\substack{u=pr_1\circ v \\ \omega(u)>0}}\omega(u)\geq A^H_\mathrm{min}(L,X).
	\end{align*}
	We then get the inequality by taking the supremum over all possible $L$'s.
	
	The case $H_1(Q';\R)=0$ is proven in much the same way. Indeed, exactness of $X'$ along with $H_1(Q';\R)=0$ ensures that we also have $(\omega\oplus\omega')(v)=\omega(u)$. Furthermore, the vanishing of the first Betti number ensures that $H_1(Q\times Q';\R)\to H_1(X\times X';\R)$ vanishes if and only if $H_1(Q;\R)\to H_1(X;\R)$ does.
\end{pr*}

\subsection{Finiteness of the capacities} \label{subsec:finiteness-capacities}

Having enunciated the main properties of our capacities, we now explain how one can get the first part of Theorem~\ref{thm:conjecture} from their finiteness.

\begin{prop}
	\label{prop:finiteness_to_conjecture}
	Let $L$ be a closed connected manifold. There exists $R>0$ such that $c^{'}_{\mathrm{all}}(D^*_RL)$ is finite {if and only if} Conjecture \ref{conj: r Vit} holds.
\end{prop}

\begin{proof}
	Suppose that $c^{'}_{\mathrm{all}}(D^*_RL)$ is finite, and set $C:= c'_{\mathrm{all}}(D^*_R L)/R$. It follows from Property~(i) of Lemma~\ref{lem:c_basic-prop} that
	\begin{align*}
		c'_{\mathrm{all}}(D^*_r L)=\frac{r}{R}c'_{\mathrm{all}}(D^*_R L)=Cr.
	\end{align*}
	Here, we have made use of the fact that $(D^*_r L,\omega_0)$ is symplectomorphic to $(D^*_{r/a}L,a\omega_0)$ via the map $(q,p)\mapsto (q,a p)$. Note that Property~(ii) of Lemma~\ref{lem:c_basic-prop} implies that our capacity is invariant under symplectomorphisms. Let $K$ be a $H$-rational Lagrangian in $D^*_rL$, with $H$-rationality constant $\tau_K$, such that $(\pi_K)_*$ is not surjective on $H_1(\,\cdot\,;\R)$, i.e. $K$ is in $\mathscr{L}_K^{'}(D^*_rL)$. Then
	\begin{equation*}
		\tau_K= A^H_\mathrm{min}(K,D^*_rL)\leq c'_{\mathrm{all}}(D^*_r L)=Cr,
	\end{equation*}
	which is what we wanted to show. \par
	
	On the other hand, if Conjecture~\ref{conj: r Vit} holds, then $A^H_\mathrm{min}(K,D^*_RL)\leq CR$ for all {\HR} Lagrangians in $D^*_RL$ such that $(\pi_K)_*$ is not surjective on $H_1(\,\cdot\,;\R)$. Since $A^H_\mathrm{min}(K,D^*_RL)=0$ when $K$ is $H$-irrational, we directly get a bound on the capacity.
\end{proof}

Therefore, proving Theorem~\ref{thm:conjecture} reduces to proving finiteness of some capacity in cotangent bundles. In general, this turns out to be nontrivial, since even $c_{\T^n}(X)$~---~the best-behaved version of our capacities~---~is only well understood when $X$ is a convex or concave toric domain, which is far from the case we need. We will explore this further down, but we already note some interesting cases where finiteness is achievable. \par

\begin{prop} \label{prop:finiteness-products}
	If $c_{Q\times Q'}(D^*_R(Q\times Q'))<\infty$, then we have that $c_{Q}(D^*_RQ)<\infty$ and $c_{Q'}(D^*_RQ')<\infty$.
\end{prop}
\begin{pr*}
	It follows from Lemma~\ref{lem:c_products} that
	\begin{align*}
		c_{Q}(D^*_RQ)\leq c_{Q\times Q'}(D^*_RQ\times D^*_RQ').
	\end{align*}
	But $D^*_RQ\times D^*_RQ'$ embeds symplectically in $D^*_{2R}(Q\times Q')$ and that embedding is a homotopy equivalence. The proposition then follows directly from Property~(ii) of Lemma~\ref{lem:c_basic-prop}, since finiteness for some $R>0$ implies finiteness for every $R>0$ by Property~(i) of that lemma.
\end{pr*}

Note that if $L$ is a displaceable Lagrangian in a tame symplectic manifold, $A_{\min}^H(L)$ is a lower bound for its displacement energy~---~this follows from Chekanov's famous estimate~\cite{Chekanov1998}. In particular, $c_\mathrm{all}(B^{2n})$ is bounded by the displacement energy of $B^{2n}$, and thus it is finite. Zhou~\cite{Zhou2020} proved a broad generalization of this result using a truncated version of Viterbo's transfer map. \par

\begin{thm}[\cite{Zhou2020}] \label{thm:Zhou}
	Let $X$ be a Liouville domain with $SH(X)=0$. We have that $c_\mathrm{all}(X)<\infty$.
\end{thm}

Note that $SH(D^*_RL)\neq 0$ because of Viterbo's isomorphism~\cite{Viterbo1999}. Therefore, Zhou's theorem never directly implies Theorem~\ref{thm:nbhd-of-H-exactness}. However, in some cases, we still manage to compare $c_Q(D^*L)$ to $c_\mathrm{all}(X)$ as we shall see below. \par

\begin{rem} \label{rem:Zhou-generality}
	Zhou actually works with homotopy~---~not homology like us~---~and allows for the possibility of weakly exact Lagrangians. He also allows some nonexact Liouville domains, but it will not be needed here. Therefore, his result is actually more general than what is cited here.
\end{rem}

\begin{thm} \label{thm:products}
	Let $L_1$, ..., $L_k$ be closed manifolds with $\dim H_1(L_i;\R)=1$ such that $L_i$ embeds as a Lagrangian in a Liouville domain $W_i$ such that $SH(W_i)=0$. Let $L_0$ be a closed manifold with $H_1(L_0;\R)=0$. We have that
	\begin{align*}
		c^{'}_{\mathrm{all}}\left(D^*_{r_0}L_0\times\dots\times D^*_{r_k}L_k\right)<\infty
	\end{align*}
	for small enough $r_i$'s. Moreover, the same is true for $c_Q$ for a closed connected manifold $Q$ with $\dim H_1(Q;\R)\leq k.$ 
\end{thm}

\begin{proof}[Proof of Theorem~\ref{thm:products}]
	We start with the case $L_0$ is a point and remove it from the notation for now. Let $L$ be a Lagrangian in the product of codisk bundles $X$ diffeomorphic to $Q$ and such that $\omega(H_2(X,L))\neq 0$. Further suppose that $f_*:H_1(Q;\R)\rightarrow H_1(X;\R)$ is not surjective. 
	Denote by $$p_i:L_1\times\dots\times L_k\to L_1\times\dots\times\widehat{L_i}\times\dots\times L_k$$ the projection onto the product of all $L_j$'s except $L_i$ and by $\Pi_i$ its lift between products of cotangent bundles. \par
	
	Since $f_*$ is not surjective, its image $I$ is a vector subspace of $H_1(X;\R)$ of dimension less than or equal to $k-1$. Therefore, there exists at least one $i\in\{1,\dots,k\}$ such that, under the identification $H_1(X;\R)\cong\R^k$, the projection $\pi_i:\R^k\rightarrow R_i=\{x\in\R^k\,|\,x_i=0\}$ is injective on $I$. Note that $(\Pi_i)_*$ is naturally identified with $\pi_i$. Therefore, {for that $i$,} we have that $\ker f_*=\ker(\Pi_i\circ f)_*$.\par
	
	Take $r_i$ small enough so that $D^*_{r_i}L_i$ symplectically embeds in $W_i$ as in the statement of the theorem. Denote by $$\Psi_i:X\hookrightarrow X'_i:=D^*_{r_1}L_1\times\dots\times W_i\times \dots\times D^*_{r_{k}}L_{k}$$ the symplectic embedding it induces and by $\om_i'$ the symplectic form on $X_i'$. Note that $\ker(\Pi_i\circ f)_*=\ker(\Psi_i\circ f)_* $. \par
	
	Suppose $u'\in H_2(X_i',L;\Z)$ such that $\om'_i(u')\neq0$. This implies that $\del u'$ is in $\ker(\Psi_i\circ f)_*=\ker f_*$ in $H_1(Q;\R$). In particular, there is some $N\in\Z$ and some $u\in H_2(X,L;\Z)$ such that $\del u= N\del u'$. {Furthermore, some diagram chasing gives that $N$ must divide the order $B_i$ of the torsion part of $H_i(D^*_{r_i}L_i)$.} Moreover, since $\Psi_i$ is an exact symplectic embedding, we have $\om(u)=N\om'(u')$. Hence
	\begin{align*}
		A^H_{\min}(L,X)\leq  B_i A^H_{\min}(L,X'_i)\leq B_i c_{\mathrm{all}}(X'_i)
	\end{align*}
	Taking the supremum over all possible $L$, we thus get
	\begin{align*}
		c^{'}_Q(X)\leq\max_i  B_i c_{\mathrm{all}}(X'_i),
	\end{align*}
	and finiteness follows from Zhou's theorem. \par
	
	When $L_0$ is not a point, one can check that the entire proof above follows in the same way.
\end{proof}

\subsection{Proof of Theorem~\ref{thm:nbhd-of-H-exactness}}
\label{subsec:proof-nbhd-H-exactness}
For completeness, we now give a quick proof of how Theorem~\ref{thm:nbhd-of-H-exactness} follows from Conjecture~\ref{conj: r Vit} when it holds.
\begin{proof}[Proof of Theorem~\ref{thm:nbhd-of-H-exactness}]
	For all sufficiently small $r>0$, there is a Weinstein neighbourhood $W_r(L)\cong D_r^*L$ of $L$ in $M$. Let $L'\in \m L(L,\tau)$, and suppose that the map $(\pi|_{L'})_*:H_1(L';\R)\rightarrow H_1(L;\R)$, induced by the restriction of the projection $\pi:D^*L\rightarrow L$ to $L'$ is not surjective. Then, there is some positive integer $k$ such that $L'$ has rationality constant $k\tau$ in $\m W_r(L)$. If Conjecture~\ref{conj: r Vit} holds, then we have that $k\tau\leq Cr$ for some $C=C(L)>0$. We thus get a contradiction if $r<\tau/C$. \par

	For such $r$, $(\pi|_{L'})_*:H_1(L';\R)\rightarrow H_1(L;\R)$ must thus be surjective whenever $L'\subseteq \m W_r(L)$. Since $H_1(L';\R)\cong H_1(L;\R)$, this means that it is an isomorphism, and the result follows from Lemma~\ref{lemma:central}.
\end{proof}

We remark that using the methods above and the $c^{0}$ variant of the capacities, we can show the following interesting {partial result}.

\begin{thm} \label{thm:nbhd-nontrivial-projection}
	Let $L$ be a Lagrangian submanifold in $M$. Suppose that, as an abstract manifold, $L$ admits a Lagrangian embedding in a Liouville domain $W$ with $SH(W)=0$. For every $\tau\geq 0$, there exists a Weinstein neighbourhood $\m W(L)$ of $L$ in $M$, such that if $L' \in \m L(\tau)$ is included in $\m W(L)$, then the map $\pi_{*} : H_1(L';\R)\to H_1(L;\R)$ induced by the projection is nonzero. 
\end{thm}

\begin{proof}[Sketch of proof] 
	By adapting the proof of Proposition \ref{prop:finiteness_to_conjecture}, one can see that finiteness of $c_{\mathrm{all}}^0(D^*_R L)$ for some $R>0$ is equivalent to the following version of Conjecture~\ref{conj: r Vit}: if $K$ is an $H$-rational Lagrangian inside $D^*_r L$ such that the map $(\pi|_K)_*: H_1(K;\R) \to H_1(L;\R)$ is zero, then the $H$-rationality constant $\tau_K$ of $K$ satisfies $\tau_K \leq C r$ for a constant $C=C(L)$. \par
	
	However, the finiteness of $c_{\mathrm{all}}^0(D^*_R L)$ for some $R>0$ follows directly from the Weinstein neighbourhood theorem, the monotonicity of the $c^0$ capacities (Lemma~\ref{lem:c0_monotonicity}), and Zhou's Theorem. 
	
	Replicating the proof of Theorem~\ref{thm:nbhd-of-H-exactness}, one gets the existence of a neighbourhood $\m W(L)$ such that whenever $L'\in\m L(\tau)$ is in $\m W(L)$, then $H_1(L';\R)\to H_1(L;\R)$ is nonzero.
\end{proof}

\begin{rem} \label{rem:nbhd-nontrivial-projection}
	Since $L$ embeds as a Lagrangian in a Liouville domain with $SH(W)=0$, so does any nearby $L'$. The Viterbo transfer morphism then implies that $H_1(L;\R)$ and $\ H_1(L';\R)$ are nonzero (see~\cite{Ritter2013}). Therefore, $H_1(L';\R)\to H_1(L;\R)$ being nonzero is consistent with the hypotheses on $L$ and $L'$.
\end{rem}

\subsection{Conjecture \ref{conj: r Vit} and covering spaces} \label{subsec:conj-r-Vit_corverings}
The condition on the Lagrangian $K\subset D^*_rL$ in Conjecture \ref{conj: r Vit}, i.e. non-surjectivity of $(\pi_K)_*$, while technical, has the advantage of behaving well with covering spaces. This is highlighted by the following proposition, which proves the second part of Theorem \ref{thm:conjecture}.

\begin{prop}\label{prop: cover}
	Suppose that Conjecture \ref{conj: r Vit} holds for a closed connected manifold $L'$, then it holds for all $L$ such that there exists a covering map $\pi:L'\ra L$.
\end{prop}

Before we prove Proposition \ref{prop: cover}, we require the following lemma from group theory. For a group $G$, denote by $\mrm{Ab}(G) = G/[G,G]$ its abelianization.

\begin{lem}\label{lma: ab rank}
	Suppose $H$ is a subgroup of finite index in a group $G.$ Then the inclusion $i:H \to G$ induces a surjection \[\Ab(H) \otimes \Q \to \Ab(G) \otimes \Q .\]
\end{lem}

\begin{proof}
	Observe that $\rank \Ab(G) = \dim_{\Q} \Hom(G,\Q).$ {It thus suffices to prove that the restriction map \[r: \Hom(G,\Q) \to \Hom(H,\Q)\] is injective.} Suppose $\nu \in \ker(r),$ i.e.\ $\nu|_H \equiv e_H$. Let $H = a_1 H, a_2 H, \ldots, a_k H$ be the right cosets of $H$ in $G.$ Then $\im(\nu) = \{e_H,\nu(a_2),...,\nu(a_k)\} \subset \Q$ is a finite subset. Since $\Q$ has no non-zero elements of finite order, this yields that $\nu \equiv e_H.$ 
\end{proof}

\begin{proof}[Proof of Proposition \ref{prop: cover}]
	First, extend the covering $p: L' \to L$ to a covering $$p: W'_r = D^*_r L' \to W_r = D^*_r L$$ for the metric $g' = p^* g$ on $L'.$ Look at a connected component $K'$ of $p^{-1}(K).$ Then $p_{K'} = p|_{K'}:K' \to K$ is a covering of $K.$ Denote by $\om' = p^*\om$ and $\la' =p^*\la$ the symplectic and Liouville forms on $W'_r$, and set $\la'_{K'} = \la'|_{K'},\ \la_K = \la|_K.$ It is easy to see that $K'$ is $H$-rational in $D^*_r L'.$ Indeed, given $A \in H_2(K', W'_r),$ we have $\brat{\om', A} = \brat{\om, p_* A},$ which is an integer multiple of $\tau_K$. Hence for a suitable integer $m \geq 1,$ $\tau_{K'} = m \tau_K \geq \tau_K.$

	It is now sufficient to prove that if $(\pi_{K'})_*: H_1(K';\Q) \to H_1(L';\Q)$ is surjective, then $(\pi_{K})_*: H_1(K;\Q) \to H_1(L;\Q)$ also is. Note that as covering maps, $p$ and $p_{K'}$, induce injective homomorphisms at the level of fundamental groups. As $H_1(X) = \Ab \pi_1(X)$ for any path connected space $X$, Corollary \ref{lma: ab rank} implies that $p_*: H_1(L';\Q) \to H_1(L;\Q)$ and $(p_K)_*: H_1(K';\Q) \to H_1(K;\Q)$ are surjective. Hence if  $(\pi_{K'})_*: H_1(K';\Q) \to H_1(L';\Q)$ is surjective, then $(p \circ \pi_{K'})_* = p_* \circ (\pi_{K'})_*: H_1(K';\Q) \to H_1(L';\Q)$ is surjective. However, as $p \circ \pi_{K'} = \pi_K \circ p_{K'}$ the image of $(p \circ \pi_{K'})_* = (\pi_K \circ p_{K'})_* =(\pi_K)_* \circ (p_{K'})_*$ is contained in that of $(\pi_K)_*.$ Therefore $(\pi_K)_*$ is surjective.
\end{proof}

\subsection{Homotopy version of Theorem~\ref{thm:conjecture}} \label{subsec:proof_thm-conj-torus}
As announced in Remark~\ref{rem:conj-Virt_homotopy}, for some diffeomorphism types, a homotopy version of Theorem~\ref{thm:conjecture} holds. In order to state the result, we recall the following definition.

\begin{defn*}
	Let $H^D_2(M,L;\Z)$ be the image of $\pi_2(M,L)$ in $H_2(M,L;\Z)$ under the Hurewicz morphism, $L$ is called \emph{rational} if $\om(H^D_2(M,L;\Z)) \subset \R$ is a discrete subgroup~---~\emph{weakly exact} if it vanishes. We define the \emph{rationality constant} of $L$ to be the nonnegative generator of $\om(H^D_2(M,L;\Z))$. 
\end{defn*}

The argument in Proposition~\ref{prop: cover} implies that the homotopy version of Conjecture~\ref{conj: r Vit} is also invariant under coverings: if the rationality constant of all $K'\subseteq D^*_rL'$ with $(\pi_{K'})_*: H_1(K';\R) \to H_1(L';\R)$ nonsurjective· is uniformly bounded and $L'\to L$ is a covering, then the rationality constant of all $K\subseteq D^*_rL$ with $(\pi_{K})_*: H_1(K;\R) \to H_1(L;\R)$ nonsurjective is uniformly bounded by the same constant.

{More generally, we have the following result.}

\begin{thm}\label{thm:conjecture_torus}
	Suppose that $L$ is diffeomorphic to the product $L_0\times \T^m$ of a closed manifold $L_0$ with $H_1(L_0;\R)=0$ and a $m$-torus. Suppose that $K$ is a rational Lagrangian in $D^*_rL$ with rationality constant $\tau_K$ and such that $(\pi_K)_*$ is not surjective, then  \[\tau_K \leq C r\] for a constant $C>0$ depending only on a choice of an auxiliary metric on $L$. Moreover, the same is true for any closed $L'$ which is covered by $L$. 
\end{thm}

\begin{proof}[Proof of Theorem~\ref{thm:conjecture_torus}]
	Suppose $f$ is a rational Lagrangian embedding of $K$ into $$D^*_a(T^m\times Q)\cong (D^*_a S^1)^m\times D^*_aQ,$$ with rationality constant $\tau_K>0$. Observe that, as an open symplectic manifold, $D^*_a S^1 = (-a,a) \times S^1$ is symplectomorphic to the punctured disk $D^2(2a)' = D^2(2a) \setminus \{0\}$ of area $2a$. The symplectomorphism is given by $$\psi: D^*_a S^1 \to D^2(2a)',\quad (p,q) \mapsto (r,\theta)$$ where $\pi r^2 = a+p$ and $\theta = 2\pi q${~---~this identifies the zero section $0_{S^1} = \{0\} \times S^1$ with the circle $S(a) = \del D(a).$} In particular $(D^*_a S^1)^m\times D^*_aQ$ is symplectomorphic via $\psi^m \times \id$ to $V = (D^2(2a)')^m \times D^*_a Q.$ 
	
	Consider the induced map $f_*: H_1(K) \to H_1(V)$. If $R = \ker(f_*) = 0,$ then $K$ is {\HE} by Lemma~\ref{lemma:central}, and $\tau_K=0$; see also Remark \ref{rmk:homotopy_proof}. Suppose therefore that $R = \ker(f_*) \neq 0.$ Let \[V_i = (D^2(2a)')^{i-1} \times D^2(2a) \times (D^2(2a)')^{m-i} \times D^*_\eps Q\] for $1\leq i\leq m$. Let $g_i: V \to V_i$ denote the inclusion. We claim that $R= \ker(f_*)=\ker( (g_i \circ f)_* )$ for at least one index $i \in \{1,\ldots, m\}.$ Indeed, as the image $I \subset \Z^m$ of $f_*$ is a free abelian subgroup of rank at most $m-1,$ there exists at least one such $i$ with the projection $\pi_i: \Z^m \to Z_i=\{k \in \Z^m \;|\; k_i = 0 \}$ being injective on $I.$
	
	However $(g_i)_*: H_1(V) \to H_1(V_i)$ is naturally identified with $\pi_i.$ This means that the period groups \[\cl P_{V, K} = \brat{[\om], H^D_2(V, K;\Z)},\,\cl P_{V_i, K} = \brat{[\om], H^D_2(V_i, K;\Z)},\] coincide. In particular $K$ is rational in $V_i$ with the same rationality constant $\tau_K.$ For topological considerations, the same is true for $K$ inside $$\wh{V}_i = (D^2(2a)')^{i-1} \times \C \times (D^2(2a)')^{m-i} \times D^*_\eps Q.$$ Now, as $K\subset V_i$, the displacement energy of $K$ inside $\wh{V}_i$ is at most $2a.$ By \cite{Chekanov1998}, this yields $\tau_K \leq 2a.$
\end{proof}

Note that we directly get Proposition~\ref{cor:Vit numbers} for rational Lagrangians in $\R^n\times [-1,1]^n$ by setting $a=1$ and considering by contradiction the embedded copy of the rational Lagrangian $K$ inside $\R^n/\Z^n \times [-1,1]^n.$ Note, however, that in this case the map $f_*$ from the proof above vanishes identically, whence we may choose any index $i$ to run the argument. This implies that one may replace $[-1,1]^n$ by $X$ as in Proposition~\ref{cor:Vit numbers}.

\section{Theorem~\ref{thm:speculation-B} and the Lagrangian $C^{0}$ flux conjecture}
\label{sec:proof-theorems-From-H-exact-2-exact}
We start this section by proving Theorem~\ref{thm:speculation-B} (Section~\ref{subsec:proof-thm-from-H-exact-to-exact}). We then turn to the Lagrangian $C^{0}$ flux conjecture and examine the limit of sequences of {\HR} Lagrangians  (Section~\ref{subsec:proof_thm_limits}). Finally, we prove a result on the existence of some symplectic isotopy between nearby Lagrangians and use it to properly show Corollary~\ref{thm:C0-flux}, that is, the Lagrangian $C^{0}$ flux conjecture (Section~\ref{subsec:proof_prop_H-exactness_flux}).
We conclude the section with a proof of Proposition~\ref{prop:ham-lag} (Section \ref{subsec:proof_prop-ham-lag}).

\subsection{Proof of Theorem~\ref{thm:speculation-B}}
\label{subsec:proof-thm-from-H-exact-to-exact}
In view of Theorem~\ref{thm:nbhd-of-H-exactness}, we only need to upgrade our neighbourhood of {\HE}ness to a neighbourhood of actual exactness. In fact, we will prove the following stronger statement.

\begin{thm}
	\label{thm:from-H-exact-to-exact}
	Let $L$ be a {\HR} Lagrangian in $(M,\omega)$ and let $L'\in\LHam(L)$ be a Lagrangian included in a Weinstein neighbourhood $\m W_{r}(L)$ of size $r>0$ such that $L'$ is $H$-exact in $\m W_{r}(L)$. Then, for a maybe smaller $r$, $L'$ is exact in $\m W_{r}(L)$.

	Moreover, if the inclusion of $L$ into $M$ induces the 0-map $H_{1}(L;\bb R) \to H_{1}(M;\bb R)$, then the same result holds with $\LHam(L)$ replaced by $\m L(\tau_L)$, the space {\HR} Lagrangians with {\HR ity} constant $\tau=\tau_L$.\\
\end{thm}
Note that this also gives a version of Theorem~\ref{thm:speculation-B} for $\LSymp(L)$ whenever $H_{1}(L;\bb R) \to H_{1}(M;\bb R)$ is zero.

\medskip
To prove the statement, we first claim that under any of these assumptions, the rationality constant of $L'$ in $\m W_{r}(L)$, seen as a subset of $T^{*}L$, is a fraction of that of $L'$ in $M$.

\begin{lem} \label{lem:values_Liouville-form}
	Let $L$ and $L'$ be as above, and denote by $\Psi : D_{r}^{*}L \to \m W_{r}(L)$ a Weinstein neighbourhood of $L$.
	There exists an integer $k=k(M,L)$ such that $\lambda_0(H_1(\Psi^{-1}(L')))\subseteq \frac{\tau}{k}\Z$.
\end{lem}

This lemma, whose proof we postpone to \S~\ref{subsubsec:proof-lemma} below, directly shows that, when $\tau=0$, {\HE}ness yields exactness.

\medskip
When $\tau > 0$, we conclude by using the following additional estimate.

\begin{prop}\label{prop:estimate_taut-form}
	Let $L\hookrightarrow (D^*_rL, d\lambda_{0})$ be a Lagrangian embedding whose image $L'$ is {\HE}. We have that
	\begin{align*}
		\label{eq:estimate_taut-form}
		\forall \beta'\in H_1(L'), \qquad  |\lambda_0(\beta')|\leq r\ell^\mathrm{min}_g(\pi_*\beta') 
	\end{align*}
	where $\ell^\mathrm{min}_g(\beta)$ denotes the length of the shortest geodesic loop for $g$ in $L$ representing the class $\beta$.
\end{prop}

Indeed, we choose a basis $\{\beta'_{1}, \ldots \beta'_{m}\}$ of $H_{1}(L')$ and we fix $r' < \frac{\tau}{k\ell}$ where
\begin{align*}
	\ell = \max \{ \ell^\mathrm{min}_g(\pi_*\beta'_{i}) \,|\, 1 \leq i \leq m \} \,.
\end{align*}
The proposition above gives that, for all $i$, $|\lambda_0(\beta'_i)| \leq r' \ell < \frac {\tau}{k}$.
Because of Lemma \ref{lem:values_Liouville-form}, we then get that $\lambda_{0}$ vanishes on $H_{1}(L')$, which proves the exactness of $L'$.

\medskip
\emph{It only remains to prove the lemma and proposition above to conclude the proof of Theorem \ref{thm:from-H-exact-to-exact}.}

\subsubsection{Proof of Proposition~\ref{prop:estimate_taut-form}}
\label{subsubsec:proof-prop}

We start with the proposition. First, let us remark that when $L=\T^n$, the estimate follows directly from Eliashberg's result on the shape of subsets of $T^*\T^n$~\cite{Eliashberg1991}. With the additional hypothesis that $L$ is also contained in a Weinstein neighbourhood of $L'$, this is a result of Membrez and Opshtein~\cite{MembrezOpshtein2021}. However, as they themselves point out, there should be a proof of this result without their additional constraint using the theory of graph selectors~---~they even sketch out a proof, which we mostly follow here.

\begin{proof}[Proof of Proposition \ref{prop:estimate_taut-form}]
	In Theorem~6.1 of~\cite{PaternainPolterovichSiburg2003}, Paternain, Polterovich, and Siburg show that, for every Lagrangian submanifold $L'\subseteq T^*L$ Lagrangian isotopic to the zero-section and every fiberwise-convex neighbourhood $W$ of $L'$, there is a closed 1-form $\sigma$ of $L$ such that $\mathrm{graph}(\sigma) \subseteq W$ and $[\sigma]=[\lambda_0|_{L'}]$. However, inspecting the proof of that statement, we see that all that is truly required is the existence of a symplectic isotopy preserving fibres sending $L'$ to an exact Lagrangian submanifold admitting a graph selector~---~we refer to that paper for the definition of a graph selector. On the one hand, we have shown in Lemma~\ref{lemma:central} that {\HE} Lagrangians in $T^*L$ indeed have associated symplectic isotopies preserving fibres which send them to exact ones. On the other hand, it is now known that every exact Lagrangian submanifold of $T^*L$ admits a graph selector. This was proven using Floer theory by Amorim, Oh, and Dos Santos~\cite{AmorimOhSantos2018} and using microlocal sheaves by Guillermou~\cite{Guillermou2023}. Therefore, the result applies as is in our case.
	
	But it follows from this that
	\begin{align*}
		\big\{[\iota^*\lambda_0]\ \big|\ \iota:L\hookrightarrow D^*_r L\ \text{is {\HE}}\big\}=\big\{[\sigma]\in H^1(L;\R)\ \big|\ |\sigma|<r\big\}.
	\end{align*}
	In particular, for every {\HE} Lagrangian embedding $\iota:L\hookrightarrow D^*_r L$ and every loop $\gamma:S^1\to L$, we have that
	\begin{align*}
		|\lambda_0(\iota\circ\gamma)|<r\ell_g(\gamma),
	\end{align*}
	where $\ell_g$ denotes the length in the metric $g$. By taking the infimum over all loops representing a class $\beta=\pi_*\beta'$, we get the desired inequality.
\end{proof}

\subsubsection{Proof of Lemma~\ref{lem:values_Liouville-form}}
\label{subsubsec:proof-lemma}

Recall that $L$ is a {\HR} Lagrangian with rationality constant $\tau\geq 0$, that $\Psi : D_{r}^{*}L \to \m W_{r}(L)$ is a Weinstein neighbourhood of $L$ in $M$ of size $r>0$, and that $L' \in \m L(\tau)$ is a Lagrangian entirely contained and {\HE} in $\m W_{r}(L)$.
The lemma states that, under one of the following conditions,
\begin{enumerate}%[label=(\alph*)]
	\item $L' \in \LHam(L)$, or
	\item the map $H_{1}(i) \otimes \bb R$ induced by the inclusion $i : L' \hookrightarrow M$ vanishes
\end{enumerate}
there exists an integer $k=k(M,L)$ such that $\lambda_0(H_1(\Psi^{-1}(L')))\subseteq \frac{\tau}{k}\Z$.

For convenience, we denote by $\overline X$ the object in $T^{*}L$ corresponding to $X$ via $\Psi^{-1}$, e.g. $\Psi^{-1}(L) = \overline L$.

\begin{proof}[Proof of Lemma  \ref{lem:values_Liouville-form}]
	Fix a representative $\overline{\beta}: S^{1} \to \overline L$ of a class in $H_1(\overline L)$. Since $\overline{L'}$ is {\HE} in $D_r^*L$, the projection $\overline{L'} \hookrightarrow T^*L\to \overline L$ is a homotopy equivalence by Lemma~\ref{lemma:central}. Therefore, there exist a loop $\overline{\beta'}$ in $\overline{L'}$ and a cylinder $\overline{C}$ in $D_r^*L$ such that $\pi_{*}(\overline{\beta'})=\overline{\beta}$ and $\del \overline{C}=\overline{\beta'}\sqcup (-\overline{\beta})$. By Stokes Theorem and exactness of the 0-section $\overline L$ in $T^*L$, we thus have that
	\begin{align*}
		\omega(C) = d\lambda_{0}(\overline{C})=\lambda_0(\overline{\beta'})-\lambda_0(\overline{\beta})=\lambda_0(\overline{\beta'}) \,.
	\end{align*}
	
	In case (a), take a Hamiltonian isotopy $\{\phi_t\}_{t\in [0,1]}$ starting at identity and such that $\phi_1(L)=L'$. Then $C'(s,t):=\phi_t^{-1}(\beta'(s))$ defines a cylinder in $M$ and $C'':=C\cup_{\beta'} C'$ represents a class in $H_2(M,L)$. In particular, $\omega(C)+\omega(C')=\omega(C'')\in\tau\Z$. But note that, since $\{\phi_t^{-1}\}$ is Hamiltonian,
	\begin{align*}
		\omega(C')=\mathrm{Flux}(\{\phi_t^{-1}\})(\beta')=0 \,.
	\end{align*}
	Therefore, $\omega(C)=\lambda_0(\overline{\beta'})\in\tau\Z$, and we can take $k=1$.
	
	In case (b),
	note that $H_1(L;\R)\to H_1(M;\R)$ being zero is equivalent to the image of $H_1(L)\to H_1(M)$ being finite, since $H_1(\,\cdot\,;\R)=H_1(\,\cdot\,)\otimes\R$. By the long exact sequence of the pair $(M,L)$, this is in turn equivalent to $H_2(M,L)\to H_1(L)$ having finite cokernel, whose size we denote by $k$. 
	Then, $k\beta$ bounds some $u\in H_2(M,L)$, and we have that
	\begin{align*}
		k\lambda_0(\overline{\beta'})=k\omega(C)=\omega(u\# kC)-\omega(u)\in \tau\Z,
	\end{align*}
	because $u\# kC\in H_2(M,L')$, and $L$ and $L'$ belong to $\m L(\tau)$. Therefore, $\lambda_0|_{L'}$ must take values in $\frac{\tau}{k}\Z$.
\end{proof}

\subsubsection{Proof of Theorem~\ref{thm: BO}}
We now have all the tools required to prove Buhovsky and Ophstein's conjecture, i.e.\ Theorem~\ref{thm: BO}.

\begin{proof}
	As in the proof of \cite[Proposition 5.2]{BO2016}, for all sufficiently large $k,$ the map $\pi \circ i_k: L \to L'$ is $C^0$-close to and hence homotopic to $i: L \to T^*L',$ where $\pi: T^*L' \to L'$ is the natural projection. By Lemma \ref{lemma:central} this implies that $L_k = i_k(L) \subset T^*L'$ is $H$-exact. Therefore by Proposition \ref{prop:estimate_taut-form} we get that \[\langle [\lambda|_{L_k}], \beta'\rangle \leq C r_k || \pi_*\beta' ||,\] where $||\cdot||$ is any norm on $H^1(L';\R),$ for a suitable constant $C >0$ depending only on $L'$ and the choice of the norm and $r_k>0$ is such that $L_k \subset D^*_{r_k} L'.$ Now write $\beta' = (i_k)_*\beta$ for an arbitrary $\beta \in H_1(L;\R).$ Then \[\langle [i_k^*\lambda], \beta \rangle = \langle i_k^*[\lambda|_{L_k}], \beta \rangle = \langle [\lambda|_{L_k}],  (i_k)_*\beta \rangle \leq C r_k || \pi_* (i_k)_*\beta || = C r_k || i_* \beta ||,\] as $\pi \circ i_k$ is homotopic to $i.$ Using the fact that $i: L \to 0_{L'} \cong L'$ is a homotopy equivalence, we get that $|| i_*( - ) ||$ is a norm on $H_1(L;\R).$ This implies that the dual norm of $[i_k^*\lambda]$ satisfies \[ || [i_k^*\lambda]|| \leq C r_k\] and hence converges to $0$ as the $r_k$ can be chosen to do so for $i_k$ $C^0$-converges to $i$ whose image is $0_{L'}.$ 
\end{proof}

\subsection{Proof of Theorem~\ref{thm:limits}}
\label{subsec:proof_thm_limits}
We now turn our attention to limits of sequences of {\HR} Lagrangians. More precisely, we want to prove the following result. This will be an essential ingredient when proving Theorem~\ref{thm:C0-flux}.

\begin{thm} \label{thm:limits}
	Let $\{L_i\}$ be a sequence of {\HR} Lagrangians of a tame symplectic manifold $M$ such that
	\begin{enumerate}%[label=(\roman*)]
		\item $\{L_i\}$ Hausdorff-converges to a $n$-dimensional smooth submanifold $L$;
		\item $\inf\tau_i>0$, where $\tau_i$ denotes the {\HR ity} constant of $L_i$.
	\end{enumerate}
	Then, $L$ is itself Lagrangian.
	
	Moreover, if $L_i$ is {\HE} in a Weinstein neighbourhood $\m W(L)$ for $i$ large, then $\lim\tau_i$ exists and is the {\HR ity} constant of $L$. This is in particular the case if the $L_i$'s respect the hypotheses of Theorem~\ref{thm:speculation-B}.
\end{thm}

{By \emph{tame}, we mean that $M$ admits an almost complex structure $J$ making $g_J:=\omega(\cdot, J\cdot)$ into a complete Riemannian metric whose sectional curvature is bounded and whose injectivity radius is bounded away from zero.

	The first part of the theorem is a fairly direct application of Laudenbach and Sikorav's result on the displaceability of non-Lagrangian submanifolds~\cite{LaudenbachSikorav1994}~---~we mostly write it here for the reader's convenience. Furthermore, the second part of the theorem is very reminiscent of Theorem~1 of~\cite{MembrezOpshtein2021}~---~the proof is in fact very inspired by what appears in that paper. The strength of our result is that it applies to sequences $\{L_i=\phi_i(L)\}$ where the sequence of Hamiltonian diffeomorphisms $\{\phi_i\}$ need not $C^0$-converge.
}

\begin{proof}[Proof of Theorem~\ref{thm:limits}]
	We start with the first part of the statement: if $L_i$ converges to $L$ with $L$ smooth and $n$-dimensional and $L_i\in\m L(\tau_i)$ with $\inf\tau_i>0$, then $L$ is Lagrangian. This follows pretty directly from Laudenbach and Sikorav's result on displacement of non-Lagrangians~\cite{LaudenbachSikorav1994}. \par
	
	Indeed, suppose $L$ is not Lagrangian. Then, $L\times S^1\subseteq M\times T^*S^1$ is also not Lagrangian and its normal bundle admits a nowhere vanishing section. Therefore, it follows from~\cite{LaudenbachSikorav1994} that, for every $\epsilon>0$, there is a Hamiltonian diffeomorphism $\phi$ of $M\times T^* S^1$ such that $\phi(L\times S^1)\cap L\times S^1=\emptyset$ and with Hofer norm $||\phi||_H<\epsilon$. But then, there is a neighbourhood $U$ of $L\times S^1$ such that $\phi(U)\cap U=\emptyset$. In particular, for $i$ large enough, $\phi(L_i\times S^1)\cap (L_i\times S^1)=\emptyset$. Therefore, if $e(L_i\times S^1)$ is the displacement energy of $L_i\times S^1$, we have that
	\begin{align*}
		\epsilon\geq \limsup e(L_i\times S^1)\geq \limsup \tau_i\geq \inf\tau_i>0,
	\end{align*}
	where the second inequality follows from Chekanov's estimate on displacement energy~\cite{Chekanov1998}. We get a contradiction by taking the limit $\epsilon\to 0$. \par
	
	The second part~---~that is, for when we know that the $L_i$'s are {\HE} in $\m W(L)$ for $i$ large~---~is proved in a very similar way as Theorem~\ref{thm:from-H-exact-to-exact}. Indeed, we know that the projection $L_i\to L$ is a homotopy equivalence by Lemma~\ref{lemma:central}. In particular, any homology class $A\in H_2(M,L;\Z)$ can be obtained by gluing a class $A_i\in H_2(M,L_i;\Z)$ to a (union of) cylinder $C_i$ with $\del C=\del A_i-\del A$ and $\del A_i=(\pi|_{L_i})^{-1}(\del A)$. Here, we identify $L_i$ with its preimage in $T^*L$ under a Weinstein neighbourhood of $L$. If $L_i\subseteq D^*_{r_i}L$, then Proposition~\ref{prop:estimate_taut-form} gives
	\begin{align*}
		\lim|\lambda_0(\del A_i)|\leq \lim r_i\ell^{\min}_g(\del A)=0,
	\end{align*}
	where we have made use of the fact that we may take $\lim r_i=0$ since $\{L_i\}$ Hausdorff-converges to $L$. Therefore,
	\begin{align*}
		|\omega(A)|=\liminf |\omega(A_i)+\lambda_0(\del A_i)|=\liminf|\omega(A_i)|.
	\end{align*}
	But by rationality, $|\omega(A_i)|=n_i\tau_i$ for some $n_i\in\Z_{\geq 0}$. Since $\inf\tau_i>0$, $\{n_i\}$ must be bounded. Therefore, by passing to a subsequence if necessary, we may suppose that $n_i\equiv n$ for all $i$. Thus, if $\tau:=\liminf\tau_i$, then $\omega(A)\in\tau\Z$. In particular, $L$ is {\HR} and its rationality constant $\tau_L$ is a multiple of $\tau$. \par
	
	We now prove that this multiple must be in fact 1. Pick a base $\{A^1,\dots, A^k\}$ of the free part of $H_2(M,L;\Z)$ such that $\omega(A^j)=\tau_L$ for all $j$, and construct $A^j_i$ and $C^j_i$ as above. By the same logic as above, we may suppose that $|\omega(A_i^j)|=n^j\tau_i$ for all $i$. Therefore, we get that
	\begin{align*}
		n^j\tau_i-r_i\ell\leq\tau_L\leq n^j\tau_i+r_i\ell,
	\end{align*}
	where $\ell:=\max_j \ell_g^{\min}(\del A^j)$. Since $\inf\tau_i>0$, for $i$ large enough, $r_i\ell<\tau$, and thus we must have $\tau_L=n^j\tau$ for all $j$, i.e.\ $n\equiv n^j$. But since the projection $L_i\to L$ is a homotopy equivalence, $\{A^1_i,\dots, A^k_i\}$ is a basis of the free part of $H_2(M,L_i;\Z)$. This is only possible if $n=1$ as $|\omega(A_i^j)|=n\tau_i$ and $\omega(H_2(M,L_i;\Z))=\tau_i\Z$. This thus implies that $\tau_L=\tau$. \par
	
	Finally, to conclude that $\lim \tau_i$ must exist, note that the contrary would imply the existence of two subsequences of $\{L_i\}$~---~thus still converging to $L$~---~such that their corresponding subsequences of $\{\tau_i\}$ converge to different values. But then, both these values would need to be equal to $\tau_L$, which is not possible.
\end{proof}

\subsection{Proof of Theorem~\ref{thm:C0-flux}}
\label{subsec:proof_prop_H-exactness_flux}
We first prove one last technical proposition before moving to finally prove the Lagrangian $C^0$ flux conjecture, i.e.\ Theorem~\ref{thm:C0-flux}.

\begin{prop} \label{prop:H-exactness_flux_plus}
	Let $L$ be a {\HR} Lagrangian submanifold of $M$ with {\HR i-} ty constant $\tau$. There is some $r_0>0$ and some $C>0$ with the following property. Assume that $L'\in\m L(\tau)$ is a Lagrangian included in a Weinstein neighbourhood $\m W_{r}(L)$ of size $r\in (0,r_0]$ such that $L'$ is $H$-exact in $\m W_{r}(L)$. Then, there is a symplectic isotopy $\{\psi_t\}_{t\in [0,1]}$ of $M$ with $|\mathrm{Flux}(\{\psi_t(L')\})|\leq Cr$ such that $\psi_1(L')$ is exact in $\m W_{r}(L)$. 
\end{prop}

By $\mathrm{Flux}(\{L_t\})\in H^1(L;\R)$, we mean the Lagrangian flux of the Lagrangian isotopy $\{L_t\}$; it is defined as follow. Take $F:[0,1]\times L\to M$ such that $F(L,t)=L_t$. Then, $F^*\omega=dt\wedge \alpha_t$ for some time-dependent 1-form $\alpha_t$ on $L$, and we set $\mathrm{Flux}(\{L_t\})(\gamma):=\int_0^1 \alpha_t(\gamma)dt$ for any loop $\gamma:S^1\to L$. This is precisely the area swept by $\gamma$ through the isotopy~---~in particular, it is independent of the parametrization $F$ of $\{L_t\}$. \par

\begin{pr*}
	Denote by $V$ the image of the boundary map $H_2(M,L';\R)\to H_1(L';\R)$. Pick a complement $W$ of $V$ in $H_1(L';\R)$, and take loops $\{\gamma_1,\dots,\gamma_k\}$ which induce a basis of $W$. Similarly to Section~\ref{subsec:proof_thm_limits} above, the proof of Theorem~\ref{thm:from-H-exact-to-exact} still implies that $\lambda_0|_{L'}(V)= 0$ for $r$ small enough. Therefore, we can take $r_0$ to ensure this is true for all $r\leq r_0$. \par
	
	We divide our isotopy in two parts. First, we consider the Lagrangian isotopy $F:t\mapsto [(\alpha-1)t+1]\cdot L'$ induced by the multiplication along the fibers of $T^*L$, where $\alpha\in [0,1]$. A direct computation gives that $F^*\omega=(\alpha-1)\lambda_0|_{L'}\wedge dt$, so that the flux associated to the isotopy is $(\alpha-1)[\lambda_0|_{L'}]$. Note that, by the above paragraph, this cohomology class is in the annihilator $V^0$ of $V$, which we can identify with the dual $W^*$ of $W$ in $H^1(L';\R)=\Hom(H_1(L';\R),\R)$. \par
	
	Second, take a closed 1-form $\sigma$ on $L$ such that $\sigma(V)= 0$ and $\sigma(\pi\circ\gamma_i)=\lambda_0(\gamma_i)$ for all $i$. It exists, since the projection $L'\to L$ is a homotopy equivalence by Lemma~\ref{lemma:central}. Consider the symplectic isotopy $\{\psi_t'\}$ of $T^*L$ generated by $X$ such that $\iota_X\omega_0=-\pi^*\sigma$, where $\pi:T^*L\to L$ is the canonical projection. It is easy to check that
	\begin{enumerate}%[label=\textit{(\roman*)}]
		\item $\psi_1'(L')$ is exact in $T^*L$,
		\item if $L'\subseteq D^*_rL$, then $\psi_t'(L')\subseteq D^*_{r+|\sigma|}L$ for all $t\in [0,1]$,
		\item $\mathrm{Flux}(\{\psi_t'(L')\})=(\iota')^*\mathrm{Flux}(\{\psi_t'\})=-(\iota')^*\pi^*[\sigma]=-[\lambda_0|_{L'}]$.
	\end{enumerate}
	We have made here the slight abuse of notation of identifying $L'$ with its preimage in $T^*L$ via the Weinstein neighbourhood. Again, \textit{(iii)} implies that the flux of the isotopy is in $W^*$.
	
	The Lagrangian isotopy $\{L'_t\}$ from $L'$ to an exact Lagrangian $L''$ that we are interested in is the (smoothing of the) concatenation of Lagrangian isotopies as above. More precisely, start with $L'\subseteq D^*_r L$ and $\sigma$ as above. Then, the first half of the isotopy is given by the scaling from $L'$ to $\alpha L'$ for $\alpha=\frac{r}{r+|\sigma|}$. Note that then, $\alpha\sigma$ is a closed 1-form on $L$ having the same properties as above for the Lagrangian $\alpha L'$. We thus get from it a symplectic isotopy $\{\psi_t'\}$ with properties \textit{(i)}--\textit{(iii)} for $\alpha L'$. In particular, $\psi_t'(\alpha L')\subseteq D^*_{\alpha r+|\alpha\sigma|}L=D^*_r L$ and $\mathrm{Flux}(\{\psi_t'(L')\})= -\alpha[\lambda_0|_{L'}]$. Therefore,
	\begin{align*}
		\mathrm{Flux}(\{L'_t\})=(\alpha-1)[\lambda_0|_{L'}]-\alpha[\lambda_0|_{L'}]=-[\lambda_0|_{L'}]\in W^*,
	\end{align*}
	where we have made use of the additivity of the flux under concatenation. Furthermore, Proposition~\ref{prop:estimate_taut-form} then implies that $|\mathrm{Flux}(\{L_t'\})|\leq r\max_i \ell_g^\mathrm{min}(\gamma_i)$, and it suffices to take $C:=\max_i \ell_g^\mathrm{min}(\gamma_i)$. \par
	
	We now show how $\{L_t'\}$ comes from a symplectic isotopy of $M$~---~this is essentially Lemma~6.6 of~\cite{Solomon2013}. Note that in the splitting $H^1(L';\R)=V^*\oplus W^*$, $W^*$ corresponds to the image of the restriction homomorphism $\Psi^*:H^1(M;\R)\to H^1(\m W_r(L);\R)$ under the restriction isomorphism $H^1(\m W_r(L);\R)\to H^1(L';\R)$. Here, we make use of the fact that $L'$ is isotopic to an exact Lagrangian of $T^*L$, so that the inclusion $L'\to \m W(L)$ induces an isomorphism on cohomology. In particular, since $[\lambda_0|_{L'}]$ belongs to $W^*$, there is a closed 1-form $\theta'$ of $M$ such that $\theta'|_{L'}=\lambda_0|_{L'}+dF$ for some function $F: L'\to \R$. We then pick an extension $F':M\to \R$ of $F$ and set $\theta:=\theta'-dF'$. Taking $\{\psi_t\}$ generated by $\theta$ gives the desired symplectic isotopy in $M$.
\end{pr*}

\begin{cor} \label{cor:H-exactness_flux_plusplus}
	By taking $r_0$ smaller if necessary, we have the following. If we have that $\mathrm{Flux}(\{\psi_t(L')\})\neq 0$, then $L'$ and $\psi_1(L')$ are in different Hamiltonian isotopy class in $M$. 
	
	Moreover, if the NLC holds on $T^*L$, then $L',L''\in \m L(\tau)$ with $L',L''\subseteq \m W_r(L)$, $r\leq r_0$, are Hamitlonian isotopic in $M$ if and only if their associated isotopy to an exact Lagrangian has the same flux.
\end{cor}
\begin{pr*}
	Suppose that there is a Hamiltonian isotopy $\{\phi_t\}$ of $M$ sending $L'$ to $\psi_1(L')$. Then, the concatenation $\{L_t''\}$ of $\{\psi_t(L')\}$ and $\{\phi_t^{-1}(\psi_1(L'))\}$ is a loop, so that $\mathrm{Flux}(\{L''_t\})\in H^1(L';\tau\Z)$. Indeed, for every loop $\gamma$ of $L'$, $\mathrm{Flux}(\{L''_t\})(\gamma)\in\tau\Z$, since it is the area of a cylinder with boundary in $L'$. If we take $r_0<\frac{\tau}{C}$, then this is only possible if $\mathrm{Flux}(\{L''_t\})=0$. Since the flux of a Hamiltonian isotopy is zero, this implies the first result.
	
	If the NLC holds on $T^*L$, we get an extension $\{\psi_t\}_{t\in [0,2]}$ of $\{\psi_t\}_{t\in [0,1]}$ to a symplectic isotopy with $\psi_2(L')=L$ and same flux. Let $\{\psi'_t\}_{t\in [0,2]}$ be the corresponding isotopy for $L''$. If $L'$ and $L''$ are Hamiltonian isotopic, we can construct a loop similarly to above using that Hamiltonian isotopy, $\{\psi_t\}$ and $\{\psi'_t\}$. We then again get that the flux of this loop is zero, so that $\mathrm{Flux}(\{\psi_t(L')\})=\mathrm{Flux}(\{\psi'_t(L'')\})$. If the fluxes are the same, then extension and concatenation as above give a symplectic isotopy in $T^*L$ from $L'$ to $L''$ with zero flux. By Proposition 2.3 of~\cite{Ono2008} or Lemma~6.7 of~\cite{Solomon2013}, that isotopy must be Hamiltonian.
\end{pr*}

\paragraph{The Lagrangian $C^{0}$ flux conjecture} We now give a proof of Theorem~\ref{thm:C0-flux}. {To do so we first prove the following more technical, but stronger, version of the corollary.}

\begin{cor} \label{cor:C0-flux_plus}
	{Suppose that $L=L_0 \times L_1\times\dots\times L_k$, where $H_1(L_0;\R)=0$ and, for $i \geq 1$, $L_i$ satisfies $H_1(L_i;\R)=\R$ and admits a Lagrangian embedding in a Liouville domain $W_i$ with $SH(W_i)=0.$ Here, we allow $L_0$ to be a point or $k=0$.} The following statements are equivalent.
	\begin{enumerate}%[label=(\alph*)]
		\item The nearby Lagrangian conjecture holds in $T^*L$.
		\item Suppose that $L'$ is a Lagrangian diffeomorphic to $L$ in a symplectic manifold $M$ and that $L'$ is in the Hausdorff closure of $\LHam(L'')$ of a {\HR} Lagrangian $L''$ in $M$. Then, $L'\in\LHam(L'')$. The same holds if $\LHam(L'')$ is replaced by $\LSympZ(L'')$.
	\end{enumerate}
\end{cor}
\begin{pr*}
	Suppose that we are in Case~\textit{(a)}. The case of $\LSympZ(L'')$ follows directly from Proposition~\ref{prop:H-exactness_flux_plus} together with Theorems~\ref{thm:nbhd-of-H-exactness} and~\ref{thm:limits}. For the case of $\LHam(L'')$, take a sequence $\{L_i\}$ in that space with limit $L'$ diffeomorphic to $L$. By Theorem~\ref{thm:limits}, $L''$ is a {\HR} Lagrangian with same rationality constant as the $L_i$'s~---~the $L_i$'s respect the hypotheses of Theorem~\ref{thm:nbhd-of-H-exactness}, so that they are {\HE} in $\m W(L')$ for $i$ large. Since all $L_i$ are Hamiltonian isotopic to each other, their associated symplectic isotopy from Proposition~\ref{prop:H-exactness_flux_plus} must all have the same flux by Corollary~\ref{cor:H-exactness_flux_plusplus}. But by that proposition, that flux must tend to 0 as $L_i\to L'$. Therefore, for $i$ large, there is a symplectic isotopy in $T^*L'$ sending $L_i$ to $L'$ with zero flux; again, we suppose that the NLC holds here. By Proposition 2.3 of~\cite{Ono2008} or Lemma~6.7 of~\cite{Solomon2013}, that isotopy must be Hamiltonian, and we have closure, i.e. \textit{(a)} implies \textit{(b)}.
	
	Suppose that we are in Case~\textit{(b)}, and let $L''$ be an exact Lagrangian of $T^*L$. Then, $L_i:=\frac{1}{i}L''$ defines a sequence of Lagrangians whose Hausdorff limit is the zero-section $L$. {But note that $L_i$ is Hamiltonian isotopic to $L''$ since it is the image of the exact Lagrangian $L''$ by the Liouville flow of $T^*L$.} Since $\LHam(L'')$ is Hausdorff closed, $L\in \LHam(L'')$, i.e. the NLC holds in $T^*L$.
\end{pr*}

{The Lagrangian $C^0$ flux conjecture then follows directly.}

\begin{proof}[Proof of Theorem~\ref{thm:C0-flux}]
	For the diffeomorphism types enumerated in Theorem~\ref{thm:C0-flux}, the NLC holds. Case~\textit{(b)} of Corollary~\ref{cor:C0-flux_plus} thus implies the Hausdorff closure of $\LHam(L')$ and $\LSymp(L')$ in $\m L(L)$ whenever $L'$ is {\HR} and diffeomorphic to $L$.
\end{proof}

\begin{rem} \label{rem:equiv_NLC-flux}
	{Suppose that all exact Lagrangians of $T^*L$ are known to be {\em diffeomorphic} to the zero-section. Then, the equivalence of \textit{(a)} and \textit{(b)} in Corollary~\ref{cor:C0-flux_plus} actually proves that the NLC for $L$ and the Lagrangian $C^0$ flux conjecture for {\HR} Lagrangians diffeomorphic to $L$ are equivalent. This hypothesis is satisfied when simple homotopy type is enough to determine diffeomorphism type, e.g. when $\dim L\leq 3$, $L=S^{2k+1}/\Z_m$ for $m\geq 3$ (see~\cite{Milnor1966}), or $L=S^n$ for $n\in\{1,2,3,5,6,12,56,61\}$ (see~\cite{WangXu2017}).}
\end{rem}

\begin{rem} \label{rem:neigh_Ltau}
	If NLC holds for $T^*L$, Corollary~\ref{cor:H-exactness_flux_plusplus} actually allows us to identify a Hausdorff neighbourhood of $L$ in $\m L(\tau)$ with a neighbourhood of $(L,0)$ in $\LHam(L)\times W^*$, where we recall that $W$ is a complement of the image of the boundary map $H_2(M,L;\R)\to H_1(L;\R)$. We do not know how much this extends to a global homeomorphism.
\end{rem}

\subsection{Proofs of some applications} \label{subsec:proof_prop-ham-lag}
We now give a proof of some applications that appeared in the introduction and had not been shown yet. Before doing so, we first prove a technical lemma.

\begin{lem} \label{lem:C0-small_Ham}
	For every Lagrangian $L$, there exists $\delta>0$ with the following property. Suppose that $\psi:M\to M$ is a map such that $d_{C^0}(\Id,\psi)<\delta$ and $\psi(L)$ is Lagrangian. Then, $\psi(L)$ is {\HE} in some $\m W(L)$.
\end{lem}
\begin{pr*}
	Take a Riemannian metric $g$ on $M$ which corresponds to a Sasaki metric on $T^*L$ on a Weinstein neighbourhood $\m W(L)$. Its geodesics starting at $L$ and going to $L'=\psi(L)$ stay in $\m W(L)$ (see Lemma~A.4 of~\cite{Chasse2024b} for example). Therefore, if we assume that $\delta$ is smaller than the injectivity radius $r_\mathrm{inj}(TM|_L)$ of the Riemannian exponential of $g$ restricted to $TM|_L$, we get for every $x\in L$ a unique geodesic $\gamma_x:[0,1]\to M$ such that $\gamma_x(0)=x$, $\gamma_x(1)=\psi(x)$, and $\gamma_x([0,1])\subseteq \m W(L)$. Moreover, $\gamma_x$ smoothly depends on $x$. Therefore, $(x,t)\mapsto \gamma_x(t)$ defines a smooth homotopy in $\m W(L)$ from the inclusion $\iota:L\hookrightarrow \m W(L)$ to $\psi\iota$. Since $\iota$ is a homotopy equivalence, then so must be $\psi\iota$. In particular, $H_2(\m W(L),L')=0$, and $L'$ is $H$-exact.
\end{pr*}

Using Theorem~\ref{thm:from-H-exact-to-exact}, this already implies that Conjecture~\ref{conj:ham-lag} holds for all {\HR} Lagrangians. We now prove the stronger statement that appears in the introduction.

\begin{proof}[Proof of Proposition~\ref{prop:ham-lag}]
	As in Lemma~\ref{lem:C0-small_Ham}, take a Riemannian metric $g$ on $M$ which corresponds to a Sasaki metric on $T^*L$ on a Weinstein neighbourhood $\m W_r(L)$ contained in $K$. Let $r_\mathrm{inj}(TM|_K)$ be the injectivity radius of the Riemannian exponential of $g$ restricted to $TM|_K$, and take $\delta=\min\{r,r_\mathrm{inj}(TM|_K)\}$. For $\phi\in\Ham(M)$, we then get a homotopy $\{f_t\}$ from $f_0=\Id$ to $f_1=\phi$ using geodesics as in that lemma. Furthermore, by geodesic convexity of $\m W_r(L)$, $f_t(L)\subseteq \m W_r(L)$ for all $t$ and $f_1(L)=\phi(L)$ is {\HE} in the neighborhood. \par
	
	Pick a Hamiltonian isotopy $\{\phi_t\}$ with $\phi_1=\phi$ and denote by $c:S^1\to C^\infty(M,M)$ the loop of smooth maps given by the concatenation of $\{f_t\}$ with $\{\phi_{1-t}\}$. Let $x_0$ be a fixed point of $\phi$ such that $t\mapsto \phi_t(x_0)$ is contractible~---~this always exists by Floer theory. Then, $f_t(x_0)=x_0$ for all $t$ so that $t\mapsto c(t)(x_0)$ is a contractible loop. But if $x$ is any point in $M$, there is a path $\alpha$ from $x$ to $x_0$, so that $(t,s)\mapsto c(t)(\alpha(s))$ defines a free homotopy from the loop defined by $x$ to the one defined by $x_0$. In particular, all loops $t\mapsto c(t)(x)$ are contractible. \par
	
	Given a loop $\gamma$ of $\phi(L)$, we consider the map $g:\T^2\to M$ defined by $g(t,s)=c(t)(\gamma(s))$. This decomposes into two cylinders: $\{f_t(\gamma(s))\}$ and $\{\phi_{1-t}(\gamma(s))\}$. Since $f_t(L)\subseteq \m W_r(L)$ for all $t$ and $L$ is exact in $\m W_r(L)$, the area of the first cylinder is $\lambda_0(\gamma)$. But the area of the second cylinder is the flux of the isotopy $\{\phi_{1-t}\}$ on $\gamma$, which is zero since the isotopy is Hamiltonian. Therefore, $g$ has area $\lambda_0(\gamma)$. \par
	
	If we are in Case~\textit{(a)}, then note that the loop $t\mapsto g(t,0)$ is contractible by the above discussion, so that the area of $g$ is actually the area of a sphere. We then conclude similarly to Theorem~\ref{thm:from-H-exact-to-exact}: pick a base $\{\gamma_i\}$ of $H_1(\phi(L))^\mathrm{free}$, then the corresponding $g_i$ respect
	\begin{align*}
		\omega(\pi_2(M))\ni |\omega(g_i)|=|\lambda_0(\gamma_i)|\leq r\ell_g^{\min}(\pi_*\gamma_i)
	\end{align*}
	by Proposition~\ref{prop:estimate_taut-form}. It thus suffices to take $r$~---~and thus $\delta$~---~small enough so that $r\max_i\ell_g^{\min}(\pi_*\gamma_i)$ is smaller than the positive generator of $\omega(\pi_2(M))$. This ensures that $\lambda_0|_{\phi(L)}$ vanishes, i.e.\ $\phi(L)$ is exact in $\m W_r(L)$. \par
	
	If we are in Case~\textit{(b)}, then note that for every $\gamma$, there is some $k$ such that $\gamma^k$ is contractible in $M$. In particular, there is a homotopy $\alpha$ from $\gamma^k$ to $x_0$. Then, $(t,s)\mapsto c(t)(\alpha(s))$ defines a homotopy from $g^k$ defined to the loop $t\mapsto \phi_t(x_0)$, which is contractible by hypothesis. Therefore,
	\begin{align*}
		\lambda_0(\gamma)=\frac{1}{k}\lambda_0(\gamma^k)=\frac{1}{k}\omega(g^k)=0
	\end{align*}
	for all loops $\gamma$, i.e.\ $\phi(L)$ is again exact in $\m W_r(L)$.
\end{proof}

Note that we get a rigidity result for sequences of Hamiltonian or symplectic diffeomorphisms from Theorem~\ref{thm:limits} and Theorem~\ref{thm:C0-flux}.
\begin{cor} \label{cor:sequence-diffeo}
	Let $\{\psi_i\}$ be a sequence of symplectomorphisms with (weak) $C^0$ limit $\psi\in C^0(M,M)$, and let $L\in\m L(\tau)$. If $\psi(L)$ is a smooth $n$-submanifold, then $\psi(L)\in\m L(\tau)$.  
	
	If, furthermore, the NLC holds on $T^*L$ and
	\begin{enumerate}%[label=(\alph*)]
		\item if $\{\psi_i\}\subseteq\Ham(M)$, then $\psi(L)\in\LHam(L)$;
		\item if $\{\psi_i\}\subseteq\Symp_0(M)$, then $\psi(L)\in\LSympZ(L)$.
	\end{enumerate}
\end{cor}

Finally, we prove the local contractibility of $\LHam(L)$ under the hypothesis that Conjecture~\ref{conj: r Vit} holds for $L$.
\begin{proof}[Proof of Proposition~\ref{prop:LHam-loc-contractible}]
	Note that it suffices to prove this statement at $L$. Fix a Weinstein neighbourhood $\Psi:D^*_rL\to \m W(L)$ as given by the conclusions of Theorems~\ref{thm:nbhd-of-H-exactness} and~\ref{thm:from-H-exact-to-exact}. Then, every $L'\in\LHam(L)$ which is in $\m W(L)$ is exact in that neighbourhood. We can thus take
	\begin{align*}
		(L',t)\mapsto \Psi(t\Psi^{-1}(L'))
	\end{align*}
	to be the contraction. Indeed, exactness in $\m W(L)$ ensures that this is a Hamiltonian isotopy for all $t>0.$ In particular $\Psi(t\Psi^{-1}(L'))$ is Hamiltonian isotopic to $L'$ and hence it is in $\LHam(L)$. But exactness also implies that the projection $\Psi^{-1}(L')\to T^*L\to L$ is a homotopy equivalence~\cite{AbouzaidKragh2018}. In particular, that projection must be surjective, otherwise $H_n(L')\to H_n(L)\neq 0$ would be zero. Therefore, $L'$ being close to $L$ implies that $L$ is close to $L'$ and hence that $L'$ and $L$ are Hausdorff-close. This means that the Hausdorff limit of $\Psi(t\Psi^{-1}(L'))$ as $t\to 0$ is precisely $L$, i.e.\ the above contraction is indeed Hausdorff-continuous.
\end{proof}

\section{Further applications and discussions} \label{sec:discussion}
%\textcolor{teal}{[TODO: Add remark about Buhovsky-Opstein]}
We close the paper with some remarks and applications that were omitted from the introduction to improve readability but that might still be of interested to the more specialized reader. \par

\subsection{Variants of the Lagrangian flux conjecture}

In this section, we briefly present variants of Speculation~\ref{conj:flux} by changing either the topology, the space $\LLag(L)$, or by adding a certain incompressibility hypothesis on $L$. %The space $\LLag(L)$ can be replaced, for instance, with the space $\Man(L)$, of all submanifolds of $M$ with the same diﬀeomorphism type as $L$, and the space $\Man_n$, of all $n$-dimensional submanifolds of $M$.

\paragraph{Various closures}
Theorem~\ref{thm:C0-flux} states that, whenever the NLC holds on $T^*L$, $\LHam(L)$ and $\LSympZ(L)$ are Hausdorff-closed in $\cl{L}(L)$, the space of Lagrangians of $M$ diffeomorphic to $L$. This is already larger then the space $\LLag(L)$ of all Lagrangians that are Lagrangian isotopic to $L$, which appears in Speculation~\ref{conj:flux}. However, in certain cases, we can prove closure in even larger spaces. \par

Indeed, if $M^{2n}$ is tame, i.e.\ $M$ is nice outside some compact set, then the Hausdorff limit of a sequence in $\LSympZ(L)$ is Lagrangian whenever it is a $n$-dimensional submanifold~---~this follows from the classical result of Laudenbach and Sikorav~\cite{LaudenbachSikorav1994} (see the proof of Theorem~\ref{thm:limits} for more details). Therefore, in such a case, we actually get closure of $\LHam(L)$ and $\LSympZ(L)$ in $\Man(L)$, the space of smooth $n$-dimensional manifolds of $M$ that are diffeomorphic to $L$, in the cases where the NLC holds. \par

Looking at the proof of Theorem~\ref{thm:C0-flux} in Section~\ref{subsec:proof_prop_H-exactness_flux}, we see that~---~whatever the diffeomorphism type of $L$ is~---~if a sequence $\{L_k\}$ in $\LSympZ(L)$ converges to some Lagrangian $L'$, then $L_k$ must be {\HE} in a Weinstein neighbourhood of $L$ for $k$ large. In view of Lemma~\ref{lemma:central}, this means that $L$ and $L'$ are simply homotopy equivalent. As noted in Remark~\ref{rem:equiv_NLC-flux}, this is sometimes enough to conclude that $L$ and $L'$ are diffeomorphic. In such cases, we can thus get closure of $\LHam(L)$ and $\LSympZ(L)$ in $\Man_n$, the space of smooth $n$-dimensional manifolds of $M$. \par

To resume, the following is a consequence of the proof of Theorem~\ref{thm:C0-flux}:
\begin{cor} \label{cor:C0-flux_various-closures}
	Let $L$ be a {\HR} Lagrangian for which Conjecture~\ref{conj: r Vit} and the NLC hold, e.g.\ it is one of the diffeomorphism type appearing in Theorem~\ref{thm:C0-flux}. Then, $\LHam(L)$ and $\LSympZ(L)$ are closed in
	\begin{enumerate}%[label=(\roman*)]
		\item $\Man(L)$ if $M$ is tame;
		\item $\Man_n$ if $M$ is tame and $n=\dim L\leq 3$.
	\end{enumerate}
\end{cor}

%\begin{rem} \label{rem:flux_various-closures}
%Note that one can upgrade from \textcolor{teal}{$\LLag(L)$} to $\Man(L)$ under the additional condition that the ambient symplectic manifold is tame. 

%Furthermore, one can upgrade from $\Man(L)$ to $\Man_n$ in some contexts as, for example, if $n=2$. Indeed, any {\HE} Lagrangian in the cotangent bundle of a surface has the same diffeomorphism type as that surface (see Lemma~\ref{lemma:central} below). This is a nontrivial update: Polterovich~\cite{Polterovich1993} constructed Lagrangian tori in the cotangent bundle of any flat manifold; these tori can be made to be arbitrarily close to the zero-section. We discussed these examples in more details at the very end of Section~\ref{sec:from-HE-exact}.
%\end{rem}

\paragraph{The Lagrangian $C^1$ flux conjecture.}
A natural variant of Speculation~\ref{conj:flux} is obtained by replacing closedness in the Hausdorff topology with closedness in the $C^1$ topology. We call this the Lagrangian $C^1$ flux conjecture. \par

By $C^1$ topology, we mean the one constructed as follows. Fix a Riemannian metric $g$ on $M$. We say that a closed connected half-dimensional submanifold $N'$ is $\epsilon$-$C^1$-close to another one $N$ if $N'$ is in a tubular neighbourhood of $N$ and there is a normal vector field $\nu$ along $N$ such that $\|\nu\|<\epsilon$ and $\exp(\nu(N))=N'$. We then set
\begin{align*}
	B(N,\epsilon):=\left\{N'\ \middle|\ \text{$N'$ is $\epsilon$-$C^1$-close to $N$}\ \&\ \text{vice-versa}\right\}.
\end{align*}
The $C^1$ topology on $\Man_n$ is then the topology generated by the $B(N,\epsilon)$'s. One can easily check that this is independent of the choice of Riemannian metric. \par

With our methods, we get the following.
\begin{cor} \label{cor:C1-flux}
	Let $L$ be a {\HR} Lagrangian in a symplectic manifold $M$. Then, $\LHam(L)$ and $\LSympZ(L)$ are $C^1$-closed in $\Man_n$.
\end{cor}

There is no requirement that $M$ be tame, since the Lagrangian condition is $C^1$-closed. There is no restriction on the diffeomorphism type, because Lagrangians which are $C^1$-close to $L$ are graphs of 1-forms in $\m W(L)$, and graphs are necessarily {\HE}. Likewise, the NLC is not needed since exact graphs are Hamiltonian isotopic to the zero-section in $\m W(L)$. Note that $C^1$-close Lagrangians are necessarily diffeomorphic so that closure in $\Man_n$ is the same as closure in $\Man(L)$. \par

The Lagrangian $C^1$ flux conjecture has been studied previously by \linebreak Ono~\cite{Ono2008} and Solomon~\cite{Solomon2013} in the case when $M$ is closed or a cotangent bundle. They proved that it holds when $L$ has Maslov class zero and is unobstructed in the sense of \cite{FOOO} and when the so-called Lagrangian flux group of $L$ is discrete, respectively. When $L$ is {\HR}, the Lagrangian flux group is automatically discrete. Therefore, our improvement with regards to Solomon's result is that we allow $M$ to be open~---~otherwise, we only have proved a subcase. As for Ono's, our condition is somewhat orthogonal to his: he needs no bad disks, but we ask for a lot of them. \par

\paragraph{$H$-incompressible Lagrangians}
Just like changing the topology on $\cl{L}(L)$ allowed us to remove the restriction on the diffeomorphism type in Theorem~\ref{thm:C0-flux}, adding a topological constraint on $L$ can achieve similar goals. \par

Indeed, if $L$ is such that the boundary map $\del: H_2(M,L;\R)\to H_1(L;\R)$ vanishes, then any nearby Lagrangians in $\LLag(L)$ must be {\HE} in a Weinstein neighbourhood of one another. This is because if $L'$ is in a Weinstein neighbourhood $\m W(L'')$, with $L',L''\in \LLag(L)$, and $A\in H_2(\m W(L''),L')$, then $\omega(A)=\lambda_0(\del A)=0$. This alleviates the need to estimate our Cieliebak--Mohnke-type capacities, so that the first part of the proof of Theorem~\ref{thm:C0-flux} goes through without any restriction on the diffeomorphism type of $L$. \par

It is well known that, in this setting, every Lagrangian isotopy is actually generated by a symplectic one (see, for example, Lemma~6.6 of~\cite{Solomon2013}). This means that we do not need the {\HR ity} condition to produce a symplectic isotopy as in Proposition~\ref{prop:H-exactness_flux_plus}. We thus only need to ensure that the Lagrangian loop $\{L''_t\}$ appearing in the proof of Corollary~\ref{cor:H-exactness_flux_plusplus} has flux in $H^1(L';\Gamma)$ for some discrete group $\Gamma\subseteq\R$. But, by the topological constraint, every cylinder with boundary $L'$ Lagrangian isotopic to $L$ is real homologous to a torus, so that
\begin{align*}
	\mathrm{Flux}(\{L''_t\})(\gamma)\in \left\{\int f^*\omega\ \middle|\ f:\T^2\to M\right\}=:\Gamma_{\T^2}
\end{align*}
for every loop $\gamma$ of $L'$. Therefore, the rest of the proof goes through as long as $\Gamma_{\T^2}$ is discrete. \par

We call Lagrangians with $\del=0$ \emph{$H$-incompressible}, since the constraint is equivalent to $H_1(L;\R)\to H_1(M;\R)$ being injective, a homological variant of the usual notion of incompressibility. Note that it implies the usual notion of incompressibility whenever $\pi_1(L)$ is abelian and torsionfree, e.g.\ when $L$ is a torus. \par

\begin{cor} \label{C0-flux_incompressible}
	Let $L$ be a $H$-incompressible Lagrangian in a symplectic manifold $M$ with $\Gamma_{\T^2}$ discrete. Suppose that the NLC holds on $T^*L$. Then, $\LHam(L)$ is Hausdorff-closed in $\LLag(L)=\LSympZ(L)$.
\end{cor}

This covers more cases than Theorem~\ref{thm:C0-flux}, e.g.\ the product of a non-nullhomo- logous loop in a Riemann surface of genus at least 1 with a non-nullhomologous loop in a Riemann surface of genus at least 2. \par

\subsection{More examples satisfying the main results}\label{sec:examples}
We give a few examples where Speculation~\ref{conj:weak} follows from the results above. We will only specify a Lagrangian submanifold in a symplectic manifold without worrying about it being rational or not, since given an irrational Lagrangian submanifold $L'$ of a rational symplectic manifold, we can construct a nearby rational Lagrangian $L$ by an arbitrarily $C^1$-small perturbation. 

First, from Theorem~\ref{thm:conjecture} we know that Conjecture~\ref{conj: r Vit} holds for closed connected manifolds $L$ which are covered by $L_0\times \T^m$, where $H_1(L_0;\R) = 0$ and $m$ is allowed to be equal to zero. For example, {this includes every manifold admitting a flat metric by Bieberbach's theorem.} Curiously, such manifolds also include every mapping torus $L$ of a diffeomorphism $f\in \Diff(L_0)$ of a simply connected manifold $L_0$ that has finite order in the smooth mapping class group, i.e.\ $f^k$ is smoothly isotopic to $\id$ for some integer $k.$ Indeed, the multiplication by $k$ map $S^1 \to S^1$ pulls the bundle $L \to S^1$ to a bundle diffeomorphic to $L_0\times S^1$. Furthermore, note that the class of manifolds covered by  $L_0\times \T^m$ (where $m$ and $L_0$ are not fixed) is also closed under products, so that Conjecture~\ref{conj: r Vit} also holds for all products of the above examples.

Second, we give a few examples of Lagrangian submanifolds $L$ where Speculation~\ref{conj:weak} holds. By Theorem~\ref{thm:conjecture} and {Theorem~\ref{thm:speculation-B}}, this includes those $L$ such that $H_1(L;\R)=\R$ which admit embeddings to a Liouville domains $W$ with $SH(W) = 0.$ For instance if $TL_0\otimes\C$ is trivial and $H_1(L_0;\R)=0$, then $L=L_0\times S^1$ embeds as a Lagrangian in $\C^n$ by the Gromov--Lees $h$-principle~\cite{Gromov1970,Lees1976} and a result of Audin, Lalonde, and Polterovich~\cite{AudinLalondePolterovich1994}. In another direction, Ekholm, Eliashberg, Murphy, and Smith~\cite{EkholmEliashbergMurphySmith2013} showed that, given any 3-manifold $L_0$, $L=L_0\#(S^1\times S^2)$ embeds as a Lagrangian in $\C^3$. But, by the van Kampen theorem, $\pi_1(L)=\pi_1(L_0)*\pi_1(S^1\times S^2)$, so that $H_1(L;\R)=H_1(L_0;\R)\oplus\R$. Therefore, {Theorem~\ref{thm:speculation-B}} covers $L=L_0\#(S^1\times S^2)$ with $H_1(L_0;\R)=0$, e.g.\ $L_0$ can be a (connected sum of) lens spaces. Finally, the Lagrangian Grassmannian $L=\Lambda_n$ admits a Lagrangian embedding in $\mathrm{Sym}(\C^n)=\C^{n(n+1)/2}$ (see, for example, \cite{AudinLalondePolterovich1994}), and $\pi_1(\Lambda_n)=\Z$.

Third, recall that, by Theorem~\ref{thm:conjecture}, the class of manifolds for which Conjecture~\ref{conj: r Vit} holds is closed under covering. In particular, Speculation~\ref{conj:weak} holds for the Klein bottle. Note that the Klein bottle admits a displaceable Lagrangian embedding in $S^2\times\C$. It is obtained from the usual Lagrangian Klein bottle in $S^2\times S^2$ (see, for example, \cite{Evans2022} and \cite{adaloglouEvans2024}) by removing a point on the second copy of $S^2$ and identifying $S^2 \setminus \{pt\}$ with $\D\subseteq\C$. Again, it is displaceable, because $\D$ is. In fact, the Klein bottle can even be made to be monotone. See Appendix~\ref{sec:klein-bottles} for a standalone proof of $H$-exactness for Lagrangian Klein bottles in the cotangent bundle of a Klein bottle.
%
%	We conclude with one additional case when we can establish Speculation~\ref{conj:strong}: when $L$ is a 2-sphere or a projective plane. Indeed, any other such Lagrangian in $\m W(L)$ is then automatically exact in that neighbourhood, so there is no need for Theorems~\ref{thm:nbhd-of-H-exactness} or~\ref{thm:from-H-exact-to-exact}. Since the NLC is known to hold in $T^*S^2$~\cite{Hind2004} and $T^*\R P^2$~\cite{HindPinsonnaultWu2016}, we thus directly get the following.
%	\begin{cor}
	%		Speculations~\ref{conj:strong} {and~\ref{conj:flux}} hold for Lagrangian 2-spheres or projective planes.
	%	\end{cor}

%\textcolor{teal}{	
	% Yet another approach one could take to the above questions for $L=\bb T^2$ is based on the resolution~\cite{DimitroglouGoodmanIvrii2016} in this case of the nearby Lagrangian conjecture. However, since this does not fit in the general framework developed here, we omit them from this version of the paper. Nonetheless, in Appendix~\ref{sec:klein-bottles}, we present a standalone proof of $H$-exactness for Lagrangian Klein bottles in the cotangent bundle of a Klein bottle as the proof is short and could be of independent interest.
	%}

\subsection{Spaces of Lagrangians with fixed {\HR}ity constant}
{In this section, we derive some applications of our main results to the space $\m L(L,\tau)$ of all Lagrangians of $M$ with the diffeomorphism type of $L$ and {\HR ity} constant $\tau$. }\par

From Theorems~\ref{thm:nbhd-of-H-exactness} and~\ref{thm:from-H-exact-to-exact}, we get the following.
\begin{cor} \label{cor:LHam-open}
	Let $L$ be a {\HR} Lagrangian in a symplectic manifold, and denote by $\tau$ its {\HR ity} constant. Then, $\LSympZ(L)$ is open in $\m L(L,\tau)$ in the $C^1$ topology. If moreover Conjecture~\ref{conj: r Vit} and the NLC holds on $T^*L$, then the same holds in the Hausdorff topology.
\end{cor}
\begin{pr*}
	Note that it suffices to prove that there is an open neighbourhood of $L$ in $\m L(L,\tau)$ which is fully in $\LSympZ(L)$. Let thus $\Psi:D^*_rL\to \m W_r(L)$ be the Weinstein neighbourhood given by Proposition~\ref{prop:H-exactness_flux_plus}. Then, every graph in $\m W(L)$ must be, up to a global symplectic isotopy, exact. Since exact graphs are Hamiltonian isotopic to the zero-section, such a graph must thus be in $\LSympZ(L)$. This proves the $C^1$ case. \par
	
	For the Hausdorff case, suppose that $r$ is also small enough so that Theorem~\ref{thm:nbhd-of-H-exactness} and Proposition~\ref{prop:H-exactness_flux_plus} hold in $\m W_r(L)$. Then, any $L'\in \m L(L,\tau)$ such that $L'\subseteq \m W(L)$ must be, up to some global symplectic isotopy, exact in $\m W(L)$~---~we still denote by $L'$ its image under the isotopy. As in the proof of Proposition~\ref{prop:LHam-loc-contractible} in Section~\ref{subsec:proof_prop-ham-lag}, we note that the path $t \mapsto \Psi(t\Psi^{-1}(L'))$, $t\in [0,1]$, is continuous in the Hausdorff topology. Furthermore, it is a Hamiltonian isotopy for all $t>0$. In particular, $L$ must be in the Hausdorff closure of $\LHam(L')\subseteq \LSympZ(L')$. But $\LSympZ(L')$ is Hausdorff closed by Theorem~\ref{thm:C0-flux} and the hypotheses on $L$. Therefore, $L'\in\LSympZ(L)$.
\end{pr*}

Together with the Lagrangian flux conjectures, this yields the following result.

\begin{cor} \label{cor:components-Ltau}
	Let $L$ and $\tau$ be as above. The (path) connected components of $\m L(L,\tau)$ in the $C^1$ topology are precisely the orbits of $\Symp_0(M)$. In particular, the quotient $\m L(L,\tau)/\Symp_0(M)$ is discrete in the induced topology. If moreover Conjecture~\ref{conj: r Vit} and the NLC holds on $T^*L$, then the same holds in the Hausdorff topology.
\end{cor}

For example, this means that a $\rho$-monotone Clifford torus can never be reached from a Chekanov torus (or any monotone special torus) by a Hausdorff-continuous path in $\m L(\T^2,2\rho)$. Contrast this with the fact that all these tori are Lagrangian isotopic~\cite{DimitroglouGoodmanIvrii2016}. \par

\begin{pr*}
	Combining Corollaries~\ref{cor:C1-flux} and~\ref{cor:LHam-open}, we get that for all $L\in\m L(L,\tau)$, the orbit $\LSympZ(L)$ is both closed and open in $\m L(L,\tau)$ with the $C^1$ topology. Therefore, $\LSympZ(L)$ must be a union of connected components of $L\in\m L(L,\tau)$ by point-set topology. Since $\LSympZ(L)$ is obviously path connected, it must be both a connected component and a path connected component of $L\in\m L(L,\tau)$. The proof in the Hausdorff setting is completely analogous.
\end{pr*}

Note that, when $H_1(L;\R)\to H_1(M;\R)$ is zero, the role of $\LSympZ(L)$ in the above proof can be replaced by $\LHam(L)$. In particular, both $\LSympZ(L)$ and $\LHam(L)$ are the connected component of $\m L(\tau)$ containing $L$ in the $C^1$ topology, so they must be equal. This can be seen as a generalization that $\mathrm{Symp}_0(M)=\mathrm{Ham}(M)$ for closed manifolds with $H_1(M;\R)=0$.
\begin{cor} \label{cor:LSymp-to-LHam}
	Let $L$ be {\HR} and such that $H_1(L;\R)\to H_1(M;\R)$ is zero. There is a symplectic isotopy $\{\psi_t\}$ of $M$ such that $\psi_1(L)=L'$ if and only if there is a Hamiltonian isotopy $\{\phi_t\}$ such that $\phi_1(L)=L'$. In other words, $\LSympZ(L)=\LHam(L)$.
\end{cor}

We also have the following.
\begin{cor} \label{cor:nonHausdorff-points}
	The space $\cup_{\tau\geq 0}\m L(L,\tau)/\Symp_0(M)$ is Hausdorff in the topology induced by the $C^1$ topology. In particular, the quotient $\LLag(L)/\Symp_0(M)$ can only be non-Hausdorff at orbits corresponding to $H$-irrational Lagrangians. The same holds for $\Symp_0(M)$ replaced by $\Ham(M)$.
\end{cor}

The part on $\Symp_0(M)$ follows directly from Corollary~\ref{cor:components-Ltau}. The part with $\Ham(M)$ is a finer result that also makes use of the local description of $L$ in $\m L(L,\tau)$ given by Corollary~\ref{cor:H-exactness_flux_plusplus} below. \par

It has been proven by Ono~\cite{Ono2008} and Solomon~\cite{Solomon2013} that the quotient $\LLag(L)/\Ham(M)$ is Hausdorff in the $C^1$ topology in different settings. Most notably, they both ask that $H_1(L;\R)\to H_1(M;\R)$ be injective, which makes $L$ automatically {\HE}. Corollary~\ref{cor:nonHausdorff-points} shows the difficulty of relaxing the condition that $H_1(L;\R)\to H_1(M;\R)$ be injective: $H$-irrational Lagrangians can create non-Hausdorff points in the quotient. In fact, Theorem~\ref{thm:Counter-Example} shows that in dimension $2n\geq 6$, this always happens. That this is a problem was already mentioned by Ono in his work on the subject. \par

\subsection{Quantitative symplectic topology}
When Theorem~\ref{thm:nbhd-of-H-exactness} holds, it allows for a new measurement associated with a Lagrangian embedding $Q\hookrightarrow M$ with image $L$ and a Riemannian metric $g$ on $Q$:
\begin{align*}
	c^{e}_{(M,L)}(Q,g):=\sup\left\{r\geq 0\ \middle|\ \text{all $L'\in\LHam(L)$ in $\ W_r^g(L)$ are exact} \right\}.
\end{align*}
By writing $\m W_r^g(L)$, we want to underline that it is the image of a Weinstein neighbourhood $\Psi:D^*_r Q\to \m W_r(L)$, where the radius $r$ of the codisk bundle is computed using $g$. We write $c^e_{(M,L)}(Q,g)=0$ if $L$ has no neighbourhood of local exactness, e.g.\ for the example given by Theorem~\ref{thm:Counter-Example}. \par

Note that $c^{e}_{(M,L)}(Q,g)$ is invariant under symplectomorphisms, so it is truly a symplectic quantity. Furthermore, $c^e_{(M,L)}(Q,g)$ is bounded from above by the size of the largest Weinstein neighbourhood of $L$ in $M$, i.e.\ by the relative capacity
\begin{align*}
	c^{\m W}_{(M,L)}(Q,g):=\sup\left\{r> 0\ \middle|\ L \text{ admits a neighbourhood } \m W_r^g(L)\right\}.
\end{align*}
This can in turn be bounded in terms of Poisson bracket invariants of $L$ in $M$~\cite{MembrezOpshtein2021}. \par

\medskip
When $L=S^1$, a direct computation gives the following estimate.

\begin{cor} \label{cor:capacity_circle}
	Let $L$ be a closed curve in a surface $M$. If $L$ bounds an embedded disk, let $A$ be the smallest area of such a disk. If there are no such disks, we set $A=+\infty$. We have that
	\begin{align*}
		c^e_{(M,L)}(S^1,g_0)=\min\left\{\frac{A}{2},c^{\m W}_{(M,L)}(S^1,g_0)\right\},
	\end{align*}
	where $g_0$ is the flat metric so that $S^1$ has length 1.
\end{cor}

Note that $\frac{A}{2}$ is precisely half the radius of the largest Weinstein neighbourhood of the circle $T(A)$ enclosing area $A$ in $\C$, i.e.\ $c^e_{(\C,S^1(A))}(S^1,g_0)=\frac{1}{2}c^{\m W}_{(\C,S^1(A))}(S^1,g_0)$. \par

In general, however, it is hard to get an estimate on $c^e_{(M,L)}(Q,g)$, as it is hard to get one on the neighbourhood for which Theorem~\ref{thm:nbhd-of-H-exactness} holds. One exception to this is when $Q=K$ is the Klein bottle: in this case, the theorem holds on every Weinstein neighbourhood (see Theorem~\ref{thm:Klein-bottles_in_cotangent-bundles} below, which is equivalent to $c_K (T^*K)=0$). Therefore, the bound comes only from the proof of Theorem~\ref{thm:from-H-exact-to-exact}~---~more precisely, from Lemma~\ref{lem:values_Liouville-form} and Proposition~\ref{prop:estimate_taut-form}. In particular, we have the following bound.
\begin{cor}
	Let $L$ be a {\HR} Lagrangian Klein bottle with {\HR ity} constant $\tau$. We have that
	\begin{align*}
		c^e_{(M,L)}(K,g)\geq \min\left\{\frac{\tau}{\ell_g^{\min}(\beta)},c^{\m W}_{(M,L)}(K,g)\right\},
	\end{align*}
	where $\ell_g^{\min}(\beta)$ denotes the minimal length in $g$ of a curve representing the generator $\beta$ of the free factor of $H_1(K;\Z)=\Z\oplus\Z_2$.
\end{cor}

\begin{rem} \label{rem:variations_capacity}
	There are of course many variations of $c^e_{(M,L)}(Q,g)$ that one could take. For example, one could be interested in $c^A_{(M,L)}(Q,g)$ or $c^B_{(M,L)}(Q,g)$, the largest neighbourhood on which Speculation~\ref{conj:strong} or Speculation~\ref{conj:weak}, respectively, holds. However, if one believes in the NLC, then we should always have $c^A=c^e$. Moreover, we have not found an example where $c^B\neq c^{\m W}$. Therefore, $c^e$ seems to be the more fruitful version of the relative capacity.
\end{rem}

% --------------------------------------------------------------------------------
%  APPENDIX
% --------------------------------------------------------------------------------

\appendix
\renewcommand{\thesection}{{\Alph{section}}}
\renewcommand{\thesubsection}{\Alph{section}.\arabic{subsection}}
\titleformat{\section}[block]{\large\bfseries}{Appendix \Alph{section}:}{0.5em}{}
\titleformat{\subsection}[block]{\bfseries}{\Alph{section}.\arabic{subsection}}{0.5em}{}
%The first 2 renewcommand change the numbering in references, while the latter titleformat hard incode the proper layout of the title. I've also changed the numbering within the appendices (in the preamble) to reflect this.
	
\section{Lagrangian Klein bottles in cotangent bundles} \label{sec:klein-bottles}
In this section we present a direct proof of a particular case of Conjecture~\ref{conj: r Vit} where $L$ is the Klein bottle $K$. Even though this case is already covered by Theorem~\ref{thm:conjecture}, the methods used might be of independent interest. The proof relies on the deep fact that there is no Lagrangian Klein bottle in $\C^2$~\cite{Shevchishin2009, Nemirovski2009}. \par

%We now focus our efforts on the case of Conjecture~\ref{conj: r Vit} where $L$ is the Klein bottle $K$. {This case is already covered by Theorem~\ref{thm:conjecture}, but we give a more direct, stronger proof of it, which is of independent interest. 

\begin{thm} \label{thm:Klein-bottles_in_cotangent-bundles}
	Every Lagrangian Klein bottle in $T^*K$ is {\HE}. In other words, $c_K (T^*K)=0$.
\end{thm}
\begin{pr*}
	Let $L$ be a Lagrangian Klein bottle in $T^*K$. We equip $K$ and the 2-torus $\T^2$ with the flat metric, so that the covering $p:\T^2\to K$ is a local isometry. By rescaling if necessary, we can suppose that $L\subseteq D^*_r K$ for $r$ arbitrarily small. In particular, we may choose $r$ small enough so that there exists a Weinstein neighbourhood $\Psi:D^*_r \T^2\to \C^2$ of the standard Clifford torus $S^1\times S^1$. \par
			
	Using the flat metric on $\T^2$ and $K$, the 2:1 covering $p:\T^2\to K$ lifts to another 2:1 covering $\widetilde{p}:T^*\T^2\to T^*K$ which is also a local isometry and symplectomorphism. Therefore, $\widetilde{L}:=\widetilde{p}^{-1}(L)$ must be a (possibly disconnected) Lagrangian submanifold of $D^*_r\T^2$. Since $\widetilde{p}|_{\widetilde{L}}$ is also a 2:1 covering, $\widetilde{L}$ must either be two disconnected copies of a Klein bottle or a 2-torus. However, if the former was the case, then each connected component of $\Psi(\widetilde{L})$ would be a Lagrangian Klein bottle in $\C^2$, which does not exist~\cite{Shevchishin2009, Nemirovski2009}. Therefore, $\widetilde{L}$ must be a 2-torus. In other words, the composition
	\begin{center}
		\begin{tikzcd}
			\T^2 \arrow[r,twoheadrightarrow,"2:1"] & L \arrow[r,hookrightarrow,"i"] & T^*K
		\end{tikzcd}
	\end{center}
	admits a lift to $T^*\T^2$, but the composition
	\begin{center}
		\begin{tikzcd}
			K \arrow[r,"\sim"] & L \arrow[r,hookrightarrow,"i"] & T^*K
		\end{tikzcd}
	\end{center}
	does not. \par
			
	We now interpret these statements in algebraic terms. To do so, we first look at the fundamental groups $\pi_1(T^*K)=\langle a,b| ab=b^{-1}a\rangle$ and $\pi_1(L)=\langle a',b'| a'b'=(b')^{-1}a'\rangle$. With these presentations, the subgroups associated to the coverings $T^*\T^2\to T^*K$ and $\T^2\to L$ are those generated by $\{a^2,b\}$ and $\{(a')^2,b'\}$, respectively. Denote $i_*(a')=a^kb^\ell$ and $i_*(b')=a^mb^n$. Here, we have made use of the presentation above to conclude that any element of $\pi_1(T^*K)$ can be written in that way. Given the lifting criterion for coverings, the fact that the composition $\T^2 \to L \to T^*K$ admits a lift is equivalent to $m$ being even. Indeed, we have that
	\begin{align*}
		i_*\left((a')^2\right)=\left(i_*(a')\right)^2=a^{2k}b^{(1+(-1)^k)\ell},
	\end{align*}
	so that this element always admits a lift to $T^*\T^2$. In turn, this forces $k$ to be odd, since the composition $K \to L \to T^*K$ does not admit a lift. In particular, $k$ is nonzero. But $a$ generates the free factor and $b$ the torsion factor of $H_1(T^*K;\Z)=\Z\oplus\Z_2$ under the Hurewicz morphism (and analogously for $a'$ and $b'$ in $H_1(L;\Z)$). Therefore, $i$ induces a monomorphism $i_*:H_1(L;\Z)^{\mathrm{free}}\to H_1(T^*K;\Z)^{\mathrm{free}}$ between the free part of the homologies. But then $i_*: H_1(L;\R)\to H_1(T^*K;\R)$ is also injective. By the long exact sequence in homology, this implies that the boundary map $\del: H_2(T^*K,L;\R)\to H_1(L;\R)$ is zero. Since $\omega_0(H_2(T^*K,L))=\lambda_0(\del (H_2(T^*K,L)))$, $L$ must be {\HE}. 
\end{pr*}

\vfill
\pagebreak

\bibliographystyle{alpha}
\bibliography{Rational_exactness,lemmasrefs}
\end{document}